\documentclass[11pt]{article}
\usepackage[margin=1 in]{geometry}
\usepackage{xcolor}
\usepackage{relsize}
\usepackage{fullpage}
\usepackage{amsfonts}
\usepackage{graphicx}
\usepackage{subfig}
\usepackage{amsmath,cite}
\usepackage{accents}
\usepackage{latexsym,array,multirow,geometry}
\usepackage{tikz,multicol,enumitem,pgfplots}
\usetikzlibrary{positioning,chains,fit,shapes,calc}
\usetikzlibrary{patterns}
\usepgfplotslibrary{fillbetween}
\usepackage[colorlinks=true,allcolors=blue]{hyperref}
\usepackage[nameinlink,noabbrev,capitalize]{cleveref}
\usepackage{amssymb}
\usepackage{amsthm}
\usepackage{stmaryrd}
\usepackage{tikz-network}
\usepackage{mathtools}
\usepackage[latin1]{inputenc}
\usepackage{tikz}
\usetikzlibrary{3d}
\usepackage{tikz-3dplot}
\usetikzlibrary{shapes,arrows}
\usepackage{adjustbox}

\usepackage{amsmath,amssymb}
\usepackage{amsthm}
\usepackage{mathtools}
\usepackage{tikz-cd}
\usepackage{mathrsfs}
\usepackage{fullpage}
\usepackage{graphicx}
\usepackage{tikz}
\usepackage{dsfont}
\usepackage[english]{babel}
\usepackage{fancyhdr}
\usepackage{wrapfig}
\usepackage{color}
\usepackage{colortbl}
\usepackage{enumitem}
\usepackage{hyperref}
\usepackage{calligra}
\usepackage{multirow, multicol}
\usepackage[linesnumbered,ruled,vlined]{algorithm2e}
\usepackage{caption}
\usepackage{subcaption}
\usepackage{mathrsfs}

\newcommand{\cw}{{\mathcal W}}

\def\Prob{{\mathbb P}}
\def\row{{\rm row}}

\def\Event{{\mathcal E}}
\def\Exp{{\mathbb E}}
\def\Net{{\mathcal N}}
\def\Var{{\rm Var}}

\newcommand*{\N}{\mathbb{N}}

\newcommand*{\R}{\mathbb{R}}

\newtheorem{theorem}{Theorem}[section]
\newtheorem{corollary}[theorem]{Corollary}
\newtheorem{prop}[theorem]{Proposition}
\newtheorem{lemma}[theorem]{Lemma}
\newtheorem{conjecture}[theorem]{Conjecture}
\theoremstyle{definition}
\newtheorem{definition}[theorem]{Definition}

\newtheorem*{remark}{Remark}

\newcommand\blfootnote[1]{
    \begingroup
    \renewcommand\thefootnote{}\footnote{#1}
    \addtocounter{footnote}{-1}
    \endgroup
}

\title{
\textbf{On the optimal objective value
of random linear programs} 
}

\author{Marzieh Bakhshi\footnote{Department of Industrial \& System Engineering, University of Tennessee, Knoxville, TN, USA. Emails: mbakhsh1@vols.utk.edu, jostrows@utk.edu.} \and James Ostrowski\footnotemark[1] \and Konstantin Tikhomirov\footnote{Department of Mathematical Sciences, Carnegie Mellon University, Pittsburgh, PA, USA.
Email: ktikhomi@andrew.cmu.edu.}}

\date{\today}
\begin{document}

\maketitle
\begin{abstract}
We consider the problem of maximizing 
$\langle c,x \rangle$ subject to the constraints $Ax \leq \mathbf{1}$,
where $x\in\R^n$,
$A$ is an $m\times n$ matrix with mutually independent centered subgaussian
entries of unit variance, and $c$ is a cost vector of unit Euclidean length.
In the asymptotic regime $n\to\infty$, $\frac{m}{n}\to\infty$, and
under some mild assumptions on $c$,
we prove that the optimal objective value $z^*$ of the linear program
satisfies
$$
\lim\limits_{n\to\infty}\sqrt{2\log(m/n)}\,z^*= 1\quad
\mbox{almost surely}.
$$
In the context of high-dimensional convex geometry, our findings
imply sharp asymptotic bounds on the spherical mean width
of the random convex polyhedron $P=\{x\in\R^n:\;
Ax\leq \mathbf{1}\}$.
We provide numerical experiments
as supporting data for the theoretical predictions. 
Further, we carry out numerical studies of the limiting
distribution and the standard deviation of $z^*$.
\end{abstract}

\blfootnote{The authors' names are in alphabetical order.}

\pagenumbering{roman}
\pagenumbering{arabic}


\section{Introduction}
\subsection{Problem setup and motivation}

Consider the linear program (LP)
\begin{equation} \label{lp_max}
\begin{split}
    &\text{maximize} \  \ \langle c,x \rangle \\
    &\text{subject to }\ Ax \leq \mathbf{1} \\
    & x \in \R^n
\end{split}
\end{equation}
where the entries of the $m\times n$
coefficient matrix $A$ are mutually independent random variables,
and ${\bf 1}$ denotes the vector of all ones.
Note that $x$ is a vector of free variables. 
In this note we deal with two settings:
either $c$ is a non-random unit vector
or $c$ is random uniformly distributed on the unit Euclidean sphere.
For any realization
of $A$ the above
LP is feasible since the zero vector satisfies all the constraints.
We denote the value of the objective function of \eqref{lp_max} by $z=z(x)= \langle c,x \rangle$, the optimal objective value by $z^*$, and any optimal solution by $x^*$.

\bigskip

The main problem we would like to address here is the following:
{\it Under the assumption that the number of constraints
is significantly larger than the number of variables,
what is the magnitude of the optimal objective value?}

\bigskip

We provide theoretical guarantees on the optimal objective value
$z^*$ in the asymptotic regime $n\to\infty$, $\frac{m}{n}\to\infty$,
and carry out empirical studies which confirm that
medium size random LPs
(with the number of constraints in the order of a few thousands) closely follow theoretical predictions.
The numerical experiments further allow us to formulate
conjectures related to the distribution of the optimal objective value,
which we hope to address in future works.

\bigskip

Our interest in LP \eqref{lp_max} is twofold.
\begin{itemize}
\item[(i)] The LP \eqref{lp_max} with i.i.d random coefficients
can be viewed as
the simplest model of an ``average-case'' linear program.
LP of this form
has been considered in the mathematical literature in the context
of estimating the average-case complexity of the simplex method
\cite{B82,B87,B99} (see Subsection~\ref{litreview} below for a more
comprehensive literature review).
Despite significant progress in that direction,
fundamental questions regarding properties of random LPs
remain open \cite[Section~6]{ST04}. The results obtained in this note contribute
to better understanding of the geometry of random feasible
regions defined as intersections of independent random half-spaces.
We view our work as a step towards
investigating random LPs allowing for sparse, inhomogeneous
and correlated constraints, which
would serve as more realistic models of linear programs occurring in practice.

\item[(ii)] Assuming that the cost vector
$c$ is uniformly distributed on the unit sphere and is independent from $A$,
the quantity
$$
2\,\Exp_c\,z^*=
2\,\Exp_c\,\max\big\{\langle c,x\rangle:\;Ax\leq {\bf 1}\big\}
$$
(where $\Exp_c$ is the expectation taken with respect to the randomness of $c$)
is the {\it spherical mean
width} $\cw(P)$ of the random polyhedron $P=\{x\in\R^n:\;Ax \leq \mathbf{1}\}$.
The mean width is a classical parameter associated with a convex set;
in particular, $\cw(P)=2M(P^\circ)$,
where $P^\circ$ is the {\it polar body} for $P$ defined as the convex
hull of the rows of $A$, and $M(\cdot)$ is the average value of the {\it Minkowski
functional} for $P^\circ$ on the unit Euclidean sphere,
$$
M(P^\circ)=\Exp_c\,\|c\|_{P^\circ}
=\Exp_c\,\inf\big\{\lambda>0:\,\lambda^{-1}c\in P^\circ\big\}.
$$
The mean width plays a key role in the study of projections of convex bodies
and volume estimates
(see, in particular, \cite[Chapters~9--11]{VershyninsBook}
and \cite{Pisier}).
The existing body of work (which we revise
in Subsection~\ref{litreview})
has focused on the problem of estimating the mean width {\it up to a constant
multiple}. In this paper, we provide sharp asymptotics of $\cw(P)$
in the regime of parameters mentioned above.
\end{itemize}

\subsection{Notation}

In the course of the note we use the letter $K$ (with subscripts) for constant numbers and 
$O(\cdot)$, $o(\cdot)$, $\Omega(\cdot)$, $\omega(\cdot)$ for standard asymptotic notations.
Specifically, for positive functions $f$ and $g$ defined on positive real numbers, $f(x) = O \big( g(x) \big)$ if there exists $K>0$ and a real number $x_0$ such that
$f(x) \leq K g(x)$ for all $x \geq x_0$, and $f(x) = o \big( g(x) \big)$ if 
$\lim_{x \rightarrow \infty} \frac{f(x)}{g(x)} = 0$. Further, 
\begin{align*}
    &f(x) = \Omega \big( g(x) \big) \ \ \text{if and only if} \ \ g(x) = O \big( f(x) \big), \\
    &f(x) = \omega \big( g(x) \big) \ \ \text{if and only if} \ \ g(x) = o \big( f(x) \big). 
\end{align*}

\bigskip

\begin{definition}
A random variable $\xi$ is {\it $K$--subgaussian} if
$$
\Exp\,\exp(\xi^2/K^2)\leq 2, \ \text{for some } \ K>0. 
$$
\end{definition}
The smallest $K$ satisfying the above inequality
is the {\it subgaussian moment} of $\xi$ (see, for example,
\cite[Chapter~2]{VershyninsBook}).
We will further call a random vector with mutually independent components
$K$--subgaussian if every component is $K$--subgaussian.
Examples of subgaussian variables include
Rademacher (i.e symmetric $\pm1$) random variables, and,
more generally, any bounded random variables.
The Rademacher distribution is of interest to us
since a matrix with i.i.d $\pm 1$ entries
produces a simplest model of a normalized system of constraints
with {\it discrete} bounded coefficients.

\subsection{Theoretical guarantees}

In this paper, we study {\it asymptotics}
of the optimal objective value of \eqref{lp_max}
in the regime when the number of free variables $n$ tends to infinity.
We will further assume that the number of constraints
$m=m_n$ is infinitely large compared to $n$, i.e
$\lim\limits_{n\to\infty}\frac{m_n}{n}=\infty$
(while the setting $\frac{m_n}{n}=O(1)$ is
of significant interest, it is not covered by our argument).

\medskip

The main
result of this note is 
\begin{theorem}[Main result: asymptotics of $z^*$ in subgaussian setting]\label{th main nongauss}
Let $K>0$.
Assume that $\lim\limits_{n\to\infty}\frac{m_n}{n}=\infty$,
and that for each $n$, the entries of the $m_n\times n$ coefficient matrix $A=A(n)$
are mutually independent
centered $K$--subgaussian variables of unit variances.
Assume further
that the non-random unit
cost vectors $c=c(n)\in\R^n$ satisfy $\lim\limits_{n\to\infty}\log^{3/2}(m_n/n)\,\|c\|_\infty=0$.
Then for any constant $\varepsilon>0$, we have
$$
\Prob\big\{1-\varepsilon\leq \sqrt{2\log(m_n/n)}\,z^* \leq 1+\varepsilon\big\}
\geq 1-n^{-\omega(1)},
$$
and, in particular,
$\lim\limits_{n\to\infty} \sqrt{2\log(m_n/n)}\,z^* = 1$ almost surely.
\end{theorem}
\begin{remark}
The result is obtained as a combination of two separate statements.
We refer to Theorem~\ref{th main lower} and Corollary~\ref{linfty norm cor} for the lower bound on $z^*$
and Theorem~\ref{th main upper} for the upper bound.
\end{remark}

Distributional invariance under rotations 
makes analysis of {\it Gaussian} random polytopes considerably simpler. In the next theorem, we avoid any structural assumptions on the cost vectors,
and thereby remove any implicit upper bounds on $\frac{m_n}{n}$:
\begin{theorem}[Asymptotics of $z^*$ in Gaussian setting]\label{th main gauss}
Assume that $\lim\limits_{n\to\infty}\frac{m_n}{n}=\infty$
and that for each $n$, the entries of the $m_n\times n$ coefficient matrix $A=A(n)$
are mutually independent standard Gaussian variables.
For each $n$, let $c(n)$ be a non-random unit cost vector.
Then for any constant $\varepsilon>0$ we have
$$
\Prob\big\{1-\varepsilon\leq \sqrt{2\log(m_n/n)}\,z^* \leq 1+\varepsilon\big\}
\geq 1-n^{-\omega(1)}.
$$
In particular,
$\lim\limits_{n\to\infty} \sqrt{2\log(m_n/n)}\,z^* = 1$ almost surely.
\end{theorem}

\medskip

As we mentioned
before, the setting when the cost vector $c$ is uniform on the Euclidean sphere
is of special interest in the field of convex geometry
since the quantity $\cw=2\,\Exp_c\,z^*$ (where $\Exp_c$ denotes the expectation
taken with respect to $c$) is the spherical mean width of the random
polyhedron $P=P(n)=\{x\in\R^n:\;Ax \leq \mathbf{1}\}$.
As corollaries of Theorems~\ref{th main nongauss}
and~\ref{th main gauss}, we obtain
\begin{corollary}[The mean width in subgaussian setting]\label{mw cor nongauss}
Let $\lim\limits_{n\to\infty}\frac{m_n}{n}=\infty$,
let matrices $A=A(n)$ be as in Theorem~\ref{th main nongauss},
and assume additionally that $\log^{3/2}m_n=o(\sqrt{n/\log n})$.
Then the spherical mean width
$\cw(P)$ of the polyhedron $P=P(n)=\{x\in\R^n:\;Ax \leq \mathbf{1}\}$
satisfies
$$
\lim\limits_{n\to\infty} \sqrt{2\log(m_n/n)}\,\cw(P) = 2\quad\mbox{almost surely.}
$$
\end{corollary}
\begin{corollary}[The mean width in Gaussian setting]\label{mw cor gauss}
Let $m_n$ and $A=A(n)$ be as in Theorem~\ref{th main gauss}.
Then the spherical mean width
$\cw(P)$ of the polyhedron $P=P(n)=\{x\in\R^n:\;Ax \leq \mathbf{1}\}$
satisfies
$$
\lim\limits_{n\to\infty} \sqrt{2\log(m_n/n)}\,\cw(P) = 2\quad\mbox{almost surely.}
$$
\end{corollary}

\subsection{Review of related literature}\label{litreview}

\subsubsection{Convex feasibility problems}

In the setting of Theorem~\ref{th main nongauss}, existence of a feasible
vector $x$ with $\sqrt{2\log(m/n)}\,\langle c,x\rangle \geq 1-\varepsilon$
is equivalent to the statement that the intersection of random convex sets
$$
C_i:=\bigg\{y\in c^\perp:\;\langle\row_i(A),
y\rangle\leq 1
-\frac{1-\varepsilon}{\sqrt{2\log(m/n)}}\langle\row_i(A),c\rangle\bigg\},
\quad 1\leq i\leq m,
$$
is non-empty.
The sets can be naturally viewed as independent random affine subspaces
of the hyperplane $c^\perp$ and, with that interpretation,
the above question is a relative of the {\it spherical perceptron} problem,
in which the object of interest is the intersection of random affine
halfspaces of the form $\{y:\,\langle R_i,y\rangle\leq -\kappa\}$,
for independent random vectors $R_1,\dots,R_m$ and a constant parameter $\kappa\in\R$
(see \cite{Tal1,Tal2} for a comprehensive discussion).
Our proof of the lower bound on $z^*$
(equivalently, showing that $\bigcap_{\,i\leq m} C_i\neq\emptyset$)
is based on an iterative process
which falls under the umbrella of the {\it block-iterative projection methods}. 
%
%

\medskip

The problem of constructing or approximating a vector
in the non-empty intersection of a given collection of convex sets
has been actively studied in optimization literature.
The special setting where the convex sets are affine hyperplanes
or halfspaces,
corresponding to solving systems of linear equations and inequalities,
was first addressed by Kaczmarz \cite{kaczmarz1937angenaherte},
Agmon \cite{Agmon1954}, Motzkin and Schoenberg \cite{motzkin1954relaxation}.
We refer, among others, to papers
\cite{Telgen1982,aharoni1989block,SV2009,EN2011,NT2014,Chubanov2015,MNR2015,GR2015,
LW2016,JJL2017,DLHN2017,BW2018,MISNA2020,NZheng2020,GMMN2021,LGu2021,zhang2021block,morshed2022sampling,ZHanyi2022,ZL2023} for more recent developments and further references.
In particular, the work \cite{aharoni1989block}
introduced the block-iterative projection methods for solving
the feasibility problems, when each update
is computed as a function of projections onto a subset of constraints.

\medskip

In view of the numerous earlier results cited above,
we would like to epmhasize that
the goal of our construction is verifying that the intersection
$\bigcap_{\,i\leq m} C_i$
is non-empty with high probability rather than
approximating a solution to a problem which
is known to be feasible in advance.
Since we aim to get asymptotically sharp bounds on $z^*$,
we have to rely on high-precision analysis of constraint's
violations, not allowing for any significant losses in our estimates.
To our best knowledge, no existing results on the 
block-iterative projection methods addresses the problem considered here.

\subsubsection{Random linear programs}

Linear programs with either the coefficient matrix or
the right hand side generated randomly,
have been actively studied in various contexts:
\begin{itemize}

\item{\it{}Average-case analysis.} As we mentioned above, the setting of random LPs
with the Gaussian coefficient matrix (or, more generally,
a matrix with independent rows having spherically symmetric distributions)
was considered by Borgwardt \cite{B82,B87,B99}
in the context of estimating the average-case
complexity of the shadow simplex method.
The average-case complexity of the Dantzig algorithm was addressed
by Smale in \cite{Smale1983} (see also \cite{Blair1986}).

\item{\it{}Smoothed analysis.} The smoothed analysis of the simplex algorithm,
when the coefficient matrix is a sum of an arbitrary deterministic
matrix and a small Gaussian perturbation,
was first carried out by Spielman and Teng in \cite{ST04},
with improved estimates obtained in later works \cite{VerSim2009,DH18,HLZ23}.

\item{\it{}Distribution of the optimal objective value under small random perturbations.}
The setting when the coefficient matrix, the RHS or the cost vector is subject to small random perturbations was considered, in particular, in
\cite{babbar1955distributions,prekopa1966probability,klatt2022limit,liu2023asymptotic}.

\item{\it{}Random integer feasibility problems.}
We refer to \cite{chandrasekaran2014integer}
for the problem of integer feasibility in the randomized setting.

\item {\it{}The optimal objective value of LPs with random cost vectors \cite{dyer1986linear}.}

\end{itemize}

\subsubsection{The mean width of random polyhedra}

As we have mentioned at the beginning of the introduction,
the mean width of random polytopes has been studied
in the convex geometry literature,
although much of the work deals with the mean width
of the {\it convex hulls} of random vectors
i.e of the polars to the random polyhedra considered in the present work.
We refer, in particular, to \cite{DGT2009,AGP2015,GHT2016} for related results
(see also \cite{Gluskin1986,Gluskin1988}).

\medskip



A two-sided bound on the mean width in the i.i.d subgaussian setting
was derived in \cite{LPRT2005} (see also
\cite{GH2002} for earlier results);
specifically, it was shown that under the assumption that
the entries of $A$ are i.i.d subgaussian variables of zero mean 
and unit variance, the mean width of the
symmetric polyhedron $P=\{x\in\R^n:\;
Ax\leq {\bf 1}\;\mbox{and}\;-Ax\leq {\bf 1}\}$
satisfies
\begin{equation}\label{agkjnglkjnflkjnlkn}
\frac{K_1}{\sqrt{\log(m/n)}}\leq
\cw(P)\leq \frac{K_2}{\sqrt{\log(m/n)}}    
\end{equation}
with high probability,
where the constants $K_1,K_2$ depend on the subgaussian
moment of the matrix entries.
The argument in \cite{LPRT2005}
heavily relies on results which introduce suboptimal constant multiples
to the estimate of $\cw(P)$,
and we believe that it cannot be adapted to derive
an asymptotically sharp version of \eqref{agkjnglkjnflkjnlkn}
which is one of the goals of the present work.

\subsection{Organization of the paper}

We derive Theorem~\ref{th main nongauss} and Corollary~\ref{mw cor nongauss}
as a combination of two separate statements: a lower bound on $z^*$
obtained in Section~\ref{alksjfalkfjlskjnl}, and a matching upper
bound treated in Section~\ref{akjshbfakjhfbkqjwhb}.
The lower bound on $z^*$ (Theorem \ref{th main lower})
is in turn obtained by combining a moderate deviations estimate
from Proposition~\ref{aljhofuhwbfohjbfldsajfbalkj} and an iterative
block-projection algorithm for constructing a feasible solution
which was briefly mentioned before.
The upper bound on $z^*$ (Theorem~\ref{th main upper})
proved in Section~\ref{akjshbfakjhfbkqjwhb} relies on 
moderate deviations Proposition~\ref{aksjfnoifurhfoioin}
and a special version of a covering argument.

\medskip

The Gaussian setting (Theorem~\ref{th main gauss} and Corollary~\ref{mw cor gauss}) is considered in Section~\ref{section Gaussian}.
Finally, results of numerical simulations are presented in Section~\ref{numericssection}.

\medskip

The diagrams below provide a high-level structure of the proofs of the main results
of the paper.

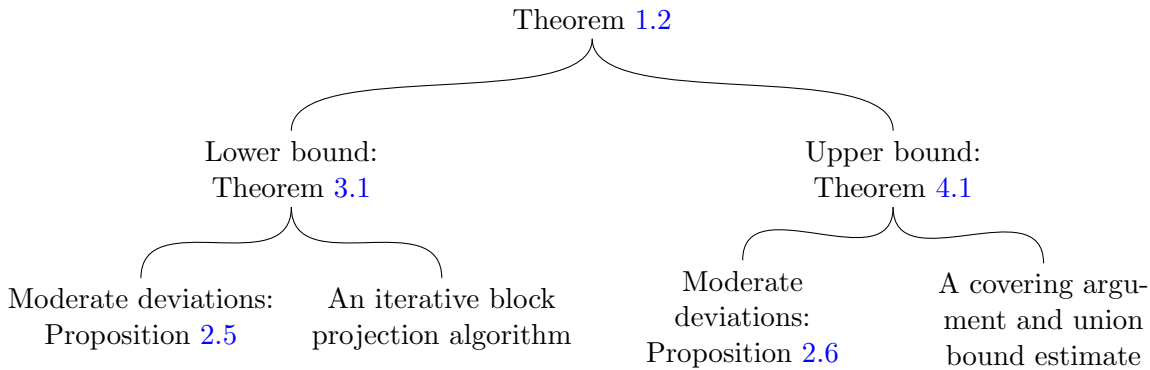
\begin{figure}[h!]
    \centering
    \begin{tikzpicture}[level/.style={sibling distance=80mm/#1 , level distance=20mm}]
        \node (root) at (0,0) {Theorem \ref{th main nongauss}}
            child[edge from parent path={
                (\tikzparentnode.south) .. controls +(0,-1) and +(0,1) .. (\tikzchildnode.north)
            }] {node[text width=8em, align=center] {Lower bound: Theorem \ref{th main lower}}
                child[edge from parent path={
                    (\tikzparentnode.south) .. controls +(0,-1) and +(0,1) .. (\tikzchildnode.north)
                }] {node[text width=10em, align=center] {Moderate deviations: Proposition \ref{aljhofuhwbfohjbfldsajfbalkj}}}
                child[edge from parent path={
                    (\tikzparentnode.south) .. controls +(0,-1) and +(0,1) .. (\tikzchildnode.north)
                }] {node[text width=9em, align=center] {An iterative block projection
                algorithm}}
            }
            child[edge from parent path={
                (\tikzparentnode.south) .. controls +(0,-1) and +(0,1) .. (\tikzchildnode.north)
            }] {node[text width=8em, align=center] {Upper bound: Theorem \ref{th main upper}}
                child[edge from parent path={
                    (\tikzparentnode.south) .. controls +(0,-1) and +(0,1) .. (\tikzchildnode.north)
                }] {node[text width=8em, align=center] {Moderate deviations: Proposition \ref{aksjfnoifurhfoioin}}}
                child[edge from parent path={
                    (\tikzparentnode.south) .. controls +(0,-1) and +(0,1) .. (\tikzchildnode.north)
                }] {node[text width=9em, align=center] {A covering argument and union bound estimate}}
            };
    \end{tikzpicture}
    \caption{Structure of the proof of Theorem \ref{th main nongauss}}
    \label{fig: thm 1.2}
\end{figure}

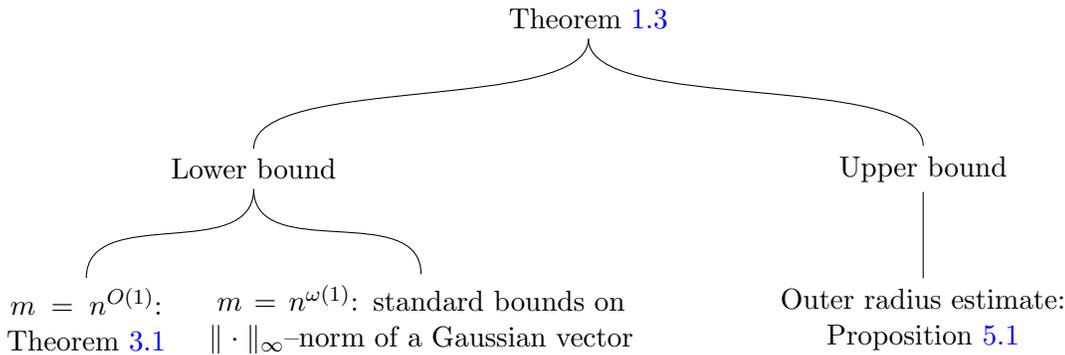
\begin{figure}[h!]
    \centering
    \begin{tikzpicture}[level/.style={sibling distance=89mm/#1 , level distance=20mm}]
        \node (root) at (0,0) {Theorem \ref{th main gauss}}
            child[edge from parent path={
                (\tikzparentnode.south) .. controls +(0,-1) and +(0,1) .. (\tikzchildnode.north)
            }] {node[text width=8em, align=center] {Lower bound  }
                child[edge from parent path={
                    (\tikzparentnode.south) .. controls +(0,-1) and +(0,1) .. (\tikzchildnode.north)
                }] {node[text width=9em, align=center] {$m=n^{O(1)}$: Theorem \ref{th main lower}}}
                child[edge from parent path={
                    (\tikzparentnode.south) .. controls +(0,-1) and +(0,1) .. (\tikzchildnode.north)
                }] {node[text width=16em, align=center] {$m=n^{\omega(1)}$: standard bounds on $\|\cdot\|_\infty$--norm of a Gaussian vector}}
            }
            child[edge from parent path={
                (\tikzparentnode.south) .. controls +(0,-1) and +(0,1) .. (\tikzchildnode.north)
            }] {node[text width=8em, align=center] {Upper bound  }
                child[edge from parent path={
                    (\tikzparentnode.south) .. controls +(0,-1) and +(0,1) .. (\tikzchildnode.north)
                }] {node[text width=10em, align=center] {Outer radius estimate: 
                Proposition \ref{prop main 3}}}
            };
    \end{tikzpicture}
    \caption{Structure of the proof of Theorem \ref{th main gauss}}
    \label{fig: thm 1.3}
\end{figure}

\section{Preliminaries}

\subsection{Compressible and incompressible vectors}

\begin{definition}[Sparse vectors]
A vector $v$ in $\R^n$ is {\it $s$--sparse} for some $s>0$
if the number of non-zero components in $v$ is at most $s$.
\end{definition}

\begin{definition}[Compressible and incompressible vectors, \cite{RV2008}]
Let $v$ be a unit vector in $\R^n$, and let $\delta,\rho\in(0,1)$.
The vector $v$ is called {\it $(\delta,\rho)$--compressible}
if the Euclidean distance from $v$ to the set of $\delta n$--sparse
vectors is at most $\rho$, that is, if there is a $\delta n$--sparse
vector $z$ such that $\|v-z\|_2\leq \rho$.
Otherwise, $v$ is {\it $(\delta,\rho)$--incompressible}.
\end{definition}
Observe that a vector $v$ is 
$(\delta,\rho)$--compressible if and only if the sum of squares
of its $\lfloor\delta n\rfloor$ largest ({\it{}by the absolute value}) components
is at least $1-\rho^2$.

\subsection{Standard concentration and anti-concentration for
subgaussian variables}

The following result is well
known, see for example \cite[Section~3.1]{VershyninsBook}.
\begin{prop}\label{alkfjbojhrbljafajkn}
Let $V$ be a vector in $\R^k$ with mutually
independent centered $K$--subgaussian components of unit variance.
Then for every $t>0$,
$$
\Prob\big\{
\big|
\|V\|_2-\sqrt{k}\big|\geq t
\big\}\leq 2\exp(-K_1 t^2),
$$
where the constant $K_1>0$ depends only on $K$.
\end{prop}

The next proposition is 
an example of Paley--Zygmund--type inequalities,
and its proof is standard (we provide the
proof for completeness).
\begin{prop}\label{alkjfnbalkfjnselfkjsnflkj}
For any $K>0$ there is $K'>0$ depending only on $K$ with
the following property.
Let $\xi$ be a centered $K$--subgaussian variable of unit variance.
Then 
$$
\Prob\big\{\xi\geq K'\big\}\geq K'.
$$
\end{prop}
\begin{proof}
Since $\xi$ is $K$--subgaussian, all moments of $\xi$ are bounded, and,
in particular, $\Exp\,\xi^4\leq \tilde K$ for some $\tilde K<\infty$
depending only on $K$.
Applying the Paley--Zygmund inequality, we get
\begin{equation}\label{alfhjbrfjhbkfjabslafj}
\Prob\Big\{\xi\geq \frac{1}{\sqrt{2}}\Big\}
+\Prob\Big\{\xi\leq -\frac{1}{\sqrt{2}}\Big\}=
\Prob\Big\{\xi^2\geq \frac{1}{2}\Big\}\geq \frac{1}{4\,\Exp\xi^4}\geq \frac{1}{4\tilde K}.
\end{equation}
Let $\varepsilon$ be the largest number in $[0,1/\sqrt{2}]$ such that
$$
\varepsilon+\sqrt{\varepsilon}+
2\int\limits_{\varepsilon^{-1/2}}^\infty \exp(-s^2/K^2)\,ds
\leq \frac{1}{16\tilde K}.
$$
We will show by contradiction
that $\Prob\{\xi\geq \varepsilon\}\geq \varepsilon$.
Assume the contrary.
Then
$
\Prob\big\{\xi\geq \frac{1}{\sqrt{2}}\big\}<\frac{1}{16\tilde K},
$
implying, in view of \eqref{alfhjbrfjhbkfjabslafj},
$$
-\Exp\,\xi{\bf 1}_{\{\xi< 0\}}\geq 
\frac{1}{\sqrt{2}}\Prob\Big\{\xi\leq -\frac{1}{\sqrt{2}}\Big\}
>\frac{1}{16\tilde K}.
$$
On the other hand,
\begin{align*}
\Exp\,\xi{\bf 1}_{\{\xi\geq 0\}}
&\leq \varepsilon+\int\limits_{\varepsilon}^{\varepsilon^{-1/2}}\Prob\{\xi\geq s\}\,ds
+
\int\limits_{\varepsilon^{-1/2}}^\infty \Prob\{\xi\geq s\}\,ds\\
&\leq 
\varepsilon+\sqrt{\varepsilon}+
2\int\limits_{\varepsilon^{-1/2}}^\infty \exp(-s^2/K^2)\,ds
\leq \frac{1}{16\tilde K}.
\end{align*}
We conclude that
$
\Exp\,\xi=\Exp\,\xi{\bf 1}_{\{\xi\geq 0\}}+\Exp\,\xi{\bf 1}_{\{\xi< 0\}}<0,
$
leading to contradiction.
\end{proof}
\begin{remark}
Instead of the requirement that the variable $\xi$ is subgaussian
in the above proposition, we could use the assumption $\Exp\,|\xi|^{2+\delta}<\infty$
for a fixed $\delta>0$.
\end{remark}

\subsection{Moderate deviations for linear combinations}

In this section we are interested in upper and lower bounds on probabilities
$$
\Prob\bigg\{\sum_{i=1}^n y_i \xi_i\geq t\bigg\},
$$
where $y=(y_1,\dots,y_n)$ is a fixed vector and $\xi_1,\dots,\xi_n$
are mutually independent centered subgaussian variables of unit variances.
Under certain assumptions on $y$ and $t$, the probability bounds
match (up to $1\pm o(1)$ term in the power of exponent) those in the Gaussian setting.
The estimates below are examples of {\it moderate deviations} results
for sums of mutually independent variables
(see, in particular, \cite[Section~3.7]{DZ1998}), and 
can be obtained with help of standard Bernstein--type bounds and
the change of measure method. However,
we were not able to locate the specific statements that we need in the literature,
and provide proofs for completeness.

\begin{prop}[Moderate deviations: upper bound]\label{aljhofuhwbfohjbfldsajfbalkj}
Let $\delta_n\in(0,1)$, $n\geq 1$, be a sequence of numbers
such that $\lim\limits_{n\to\infty}(\delta_n\, n)=\infty$.
For each $n$, let $c=c(n)$
be a non-random unit vector in $\R^n$,
and assume that for every $\nu>0$ there is $n_0(\nu)<\infty$
such that $c(n)$ is $(\nu^{-1}\delta_n,1-\nu)$--incompressible (in $\R^n$)
for all $n\geq n_0(\nu)$.
Fix any $K\geq 1$, and for each
$n\geq 1$ let $\xi_1^{(n)},\xi_2^{(n)},\dots,\xi_n^{(n)}$
be mutually independent $K$--subgaussian variables of zero mean and
unit variance.
Then for every $\varepsilon>0$ there is $n_1(\varepsilon,K)<\infty$ 
such that
$$
\Prob\bigg\{\sum_{i=1}^n c_i(n)\xi_i^{(n)}
\geq (1+\varepsilon)\sqrt{\delta_n\, n}\bigg\}
\leq \exp\big(-\delta_n\, n/2\big),\quad n\geq n_1(\varepsilon,K).
$$
\end{prop}
\begin{proof}
Throughout the proof, we assume that $K,\varepsilon>0$ are fixed.
Let $\lambda_n:=\sqrt{\delta_n\, n}$.
We have, by Markov's inequality,
\begin{align}\nonumber
\Prob\bigg\{\sum_{i=1}^n c_i(n)\xi_i^{(n)}
\geq (1+\varepsilon)\sqrt{\delta_n\, n}\bigg\}
&=
\Prob\bigg\{\exp\Big(\lambda_n\sum_{i=1}^n c_i(n)\xi_i^{(n)}\Big)
\geq \exp\big(\lambda_n(1+\varepsilon)\sqrt{\delta_n\, n}\big)\bigg\}\\
&\leq \frac{\prod\limits_{i=1}^n \Exp\,\exp(\lambda_n\, c_i(n) \xi_i^{(n)})}
{\exp(\lambda_n(1+\varepsilon)\sqrt{\delta_n\, n})}
\label{eq: ljahbljahbljshbafljh}.
\end{align}
Since $\xi_i^{(n)}$ has zero mean and unit variance,
for every $i\leq n$ we have
\begin{equation}\label{eq: alkjsnflaknalksf}
\Exp\,\exp(\lambda_n \, c_i(n) \xi_i^{(n)})
=1+\frac{1}{2}\lambda_n^2\, \big(c_i(n)\big)^2+\sum\limits_{j=3}^\infty
\frac{\big(\lambda_n\, c_i(n)\big)^j\,\Exp\,\big(\xi_i^{(n)}\big)^j}{j!}.
\end{equation}
Hence, for all $i\leq n$ such that $|\lambda_n\, c_i(n)|\leq 1$,
\begin{equation}\label{eq: kjiurqoijmnda}
\Exp\,\exp(\lambda_n\, c_i(n) \xi_i^{(n)})=1+\frac{1}{2}\lambda_n^2\, c_i(n)^2+O\big(\lambda_n^3 c_i(n)^3\big),
\end{equation}
where the implicit constant in $O(\dots)$ depends on $K$ only.
On the other hand, for every $i\leq n$
we can write (see, in particular, \cite[Section~2.5]{VershyninsBook})
$$
\Exp\,\exp\big(\lambda_n\, c_i(n) \xi_i^{(n)}\big)
\leq \exp\big(\tilde K\lambda_n^2\, c_i(n)^2\big),
$$
for some $\tilde K>0$ depending only on $K$.
Next, by the assumptions on incompressibility
of the cost vectors $c(n)$,
for every $\nu>0$ there is $n_0(\nu)<\infty$
such that for all $n\geq n_0(\nu)$,
the sum of squares of $\lfloor \nu^{-1}\delta_n\,n\rfloor$
largest (by the absolute value) components of $c(n)$
does not exceed $1-(1-\nu)^2$.
Note that every component of $c(n)$
larger than $\big(\lfloor \nu^{-1}\delta_n\,n\rfloor\big)^{-1/2}$
must necessarily be among the $\lfloor \nu^{-1}\delta_n\,n\rfloor$
largest components. Therefore, we obtain
that for all $\nu>0$,
$$
\sum_{i:\,|c_i(n)|^2> (\lfloor\nu^{-1}\delta_n\,n\rfloor)^{-1}}\,c_i(n)^2
\leq 1-(1-\nu)^2,\quad n\geq n_0(\nu).
$$
The last assertion can be written in a more compact form
using asymptotic notation as follows:
there is a sequence of numbers $\kappa_n\in(0,1]$, $n\geq 1$,
converging to zero as $n\to\infty$, such that
$$
\sum_{i:\,|\lambda_n c_i(n)|> \kappa_n} c_i(n)^2=o(1).
$$
Applying \eqref{eq: ljahbljahbljshbafljh}, \eqref{eq: alkjsnflaknalksf}, \eqref{eq: kjiurqoijmnda}, and the last relation,
we then get
\begin{align*}
\Prob&\bigg\{\sum_{i=1}^n c_i(n)\xi_i^{(n)}
\geq (1+\varepsilon)\sqrt{\delta_n\, n}\bigg\}\\
&\leq \frac{\prod\limits_{i:\,|\lambda_n c_i(n)|\leq \kappa_n} 
\big(1+\frac{1}{2}\lambda_n^2\, c_i(n)^2+O(\lambda_n^3\, c_i(n)^3)\big)
\prod\limits_{i:\,|\lambda_n c_i(n)|> \kappa_n} \exp(\tilde K\lambda_n^2 \,c_i(n)^2)}
{\exp(\lambda_n(1+\varepsilon)\sqrt{\delta_n\, n})}\\
&\leq  
 \frac{\exp\Big(\sum\limits_{i:\,|\lambda_n c_i(n)|\leq \kappa_n}\big(\frac{1}{2}\lambda_n^2\, c_i(n)^2+\kappa_n\;O(\lambda_n^2\, c_i(n)^2)\big)\Big)
\exp\Big(\tilde K \sum\limits_{i:\,|\lambda_n c_i(n)|> \kappa_n}\lambda_n^2 \, c_i(n)^2\Big)}
{\exp(\lambda_n(1+\varepsilon)\sqrt{\delta_n\, n})}\\
&=\exp\big(-\lambda_n(1+\varepsilon)\sqrt{\delta_n\, n}\big)
\exp\Big(\frac{1}{2}\lambda_n^2+o(\lambda_n^2)\Big)\\
&=\exp\Big(\big(-\frac{1}{2}-\varepsilon\big)\,\delta_n\, n+o(\delta_n\, n)\Big).
\end{align*}
The result follows.
\end{proof}

\begin{prop}[Moderate deviations: lower bound]\label{aksjfnoifurhfoioin}
Let $\delta_k\in(0,1)$, $k\geq 1$, be a sequence of numbers
such that $\lim\limits_{k\to\infty}(\delta_k\, k)=\infty$.
For each integer $k\geq 1$, let $y(k)$
be a non-random unit vector in $\R^k$,
and assume that $\lim\limits_{k\to\infty}\big(\sqrt{\delta_k\, k}\,\|y(k)\|_\infty\big)=0$.
Fix any $K\geq 1$, and for each
$k\geq 1$, let $\xi_1^{(k)},\xi_2^{(k)},\dots,\xi_k^{(k)}$
be mutually independent $K$--subgaussian variables of zero mean and
unit variance.
Then for every $\varepsilon>0$
there is $k_2(\varepsilon)<\infty$ with  
$$
\Prob\bigg\{\sum_{i=1}^k y_i(k)\xi_i^{(k)}
\geq (1-\varepsilon)\sqrt{\delta_k\, k}\bigg\}
\geq \exp\big(-\delta_k\, k/2\big),\quad k\geq k_2(\varepsilon).
$$
\end{prop}
\begin{proof}
Fix $K$ and $\varepsilon$.
We shall apply the standard technique of change of measure. 
Define $s_k:=(1-\varepsilon/4)\sqrt{\delta_k\, k}$.
For each $i\leq k$, let
\begin{equation}\label{asjhfoufyifuhboahbljhbflj}
B_{i,k}:=\Exp\, \exp(s_k\,y_i(k)\xi_i^{(k)})=1+(1+o(1))\frac{s_k^2\,(y_i(k))^2}{2},
\end{equation}
where the estimate on the right hand side is 
due to the assumptions $s_k\,\|y(k)\|_\infty=o(1)$,
$\Exp\,\xi_i^{(k)}=0$, $\Exp\,\big(\xi_i^{(k)}\big)^2=1$.
For every $i\leq k$, let $Z_i^{(k)}$ be a random variable defined via its
cumulative distribution function
$$
\Prob\big\{Z_i^{(k)}\leq \tau\big\}=\int\limits_{-\infty}^\tau \frac{\exp(s_k z)}{B_{i,k}}
d\mu_{i,k}(z),\quad \tau\in\R,
$$
where $\mu_{i,k}$ is the probability measure on $\R$ induced by $y_i(k)\xi_i^{(k)}$,
and such that $Z_1^{(k)},Z_2^{(k)},\dots,Z_k^{(k)}$
are mutually independent.
Denote by $T_k$ the set of all vectors $(z_1,\dots,z_k)$ in $\R^k$
such that $(1-\varepsilon/2)\sqrt{\delta_k\, k}
>\sum_{i=1}^k z_i> (1-\varepsilon)\sqrt{\delta_k\, k}$.
We have
\begin{align}\nonumber
\Prob&\bigg\{\sum_{i=1}^k y_i(k)\xi_i^{(k)}\geq (1-\varepsilon)\sqrt{\delta_k\, k}\bigg\}\\
&\geq
\int\limits_{T_k}
d\mu_{1,k}(z_1)\dots d\mu_{k,k}(z_k)\nonumber\\
&\geq\Big(\prod_{i=1}^k B_{i,k}\Big)\exp(-s_k(1-\varepsilon/2)\sqrt{\delta_k\, k})
\int\limits_{T_k}
\Big(\prod_{i=1}^k\frac{\exp(s_k z_i)}{B_{i,k}}\Big)\,d\mu_{1,k}(z_1)\dots d\mu_{k,k}(z_k)\nonumber\\
&=\Big(\prod_{i=1}^k B_{i,k}\Big)\exp(-s_k(1-\varepsilon/2)\sqrt{\delta_k\, k})\,
\Prob\big\{(Z_1^{(k)},\dots,Z_k^{(k)})\in T_k\big\}.
\label{eq: lakjbfjfhbwfjhbwfkjhab}
\end{align}
To estimate $\Prob\{(Z_1^{(k)},\dots,Z_k^{(k)})\in T_k\}$, we will compute the means
and the variances of $Z_1^{(k)},\dots,Z_k^{(k)}$.
Using Taylor's expansion of $\exp(s_k z)$, we get
\begin{align*}
\Exp\,Z_i^{(k)}&=\int\limits_{\R}\frac{z\exp(s_k z)}{B_{i,k}}d\mu_{i,k}(z)
=\frac{1}{B_{i,k}}
\int\limits_{\R}\sum_{j=0}^\infty \frac{s_k^j\,z^{j+1}}{j!}
d\mu_{i,k}(z)\\
&=\frac{1}{B_{i,k}}
\sum_{j=0}^\infty \frac{s_k^j\,\Exp\big(y_i(k)\xi_i^{(k)}\big)^{j+1}}{j!}
=\frac{1}{B_{i,k}}\bigg(
s_k\, y_i(k)^2+\sum_{j=2}^\infty\frac{s_k^j\, y_i(k)^{j+1}\,\Exp\,\big(\xi_i^{(k)}\big)^{j+1}}{j!}\bigg)\\
&=\frac{s_k\, y_i(k)^2\,(1+o(1))}{B_{i,k}},
\end{align*}
where for the last identity we again used that $s_k\,\|y(k)\|_\infty=o(1)$, together with the assumption
that $\xi_i^{(k)}$ is subgaussian.
Similarly, we compute
\begin{align*}
\Exp(Z_i^{(k)})^2&=\int\limits_{\R}\frac{z^2\exp(s_k z)}{B_{i,k}}d\mu_{i,k}(z)\\
&=\frac{1}{B_{i,k}}\bigg(
y_i(k)^2+\sum_{j=1}^\infty\frac{s_k^j\, y_i(k)^{j+2}\,\Exp(\xi_i^{(k)})^{j+2}}{j!}\bigg)
=\frac{y_i(k)^2\,(1+o(1))}{B_{i,k}}.
\end{align*}
Thus, in view of \eqref{asjhfoufyifuhboahbljhbflj}
and the definition of $s_k$,
$$
\Exp\,\sum_{i=1}^k Z_i^{(k)}=
(1+o(1))\sum_{i=1}^k\,s_k\, y_i(k)^2=
(1+o(1))\,s_k=
(1-\varepsilon/4+ o(1))\sqrt{\delta_k\, k},
$$
and
$$
\Var\Big(\sum_{i=1}^k Z_i^{(k)}\Big)
=\sum_{i=1}^k \Var\,Z_i^{(k)}
\leq 1+o(1).
$$
It follows from the Markov--Chebyshev inequality 
and the definition of $T_k$
that $\Prob\{(Z_1^{(k)},\dots,Z_k^{(k)})\in T_k\}=1-o(1)$,
and thus, by \eqref{eq: lakjbfjfhbwfjhbwfkjhab},
$$
\Prob\bigg\{\sum_{i=1}^k y_i(k)\xi_i^{(k)}\geq (1-\varepsilon)\sqrt{\delta_k\, k}\bigg\}
\geq (1-o(1)\,\Big(\prod_{i=1}^k B_{i,k}\Big)\exp(-s_k(1-\varepsilon/2)\sqrt{\delta_k\, k}).
$$
To complete the proof, it remains to note that,
in view of \eqref{asjhfoufyifuhboahbljhbflj},
$$
\Big(\prod_{i=1}^k B_{i,k}\Big)
\exp(-s_k(1-\varepsilon/2)\sqrt{\delta_k\, k})
=\exp\Big(\frac{1\pm o(1)}{2}s_k^2-s_k(1-\varepsilon/2)\sqrt{\delta_k\, k}\Big)
= \omega\big(\exp(-\delta_k\, k/2)\big).
$$
\end{proof}

\section{The lower bound on $z^*$ in subgaussian setting}\label{alksjfalkfjlskjnl}

The main result of this section is
\begin{theorem}[Lower bound on $z^*$]\label{th main lower}
Let $K>0$, and let $(m_n)_{n=1}^\infty$
be a sequence of positive integers with 
$\lim\limits_{n\to\infty}\frac{m_n}{n}=\infty$.
For each $n$, let $A=A(n)$ be an
$m_n\times n$ coefficient matrix with 
mutually independent
centered $K$--subgaussian entries of unit variance.
Assume further
that the non-random unit
cost vectors $c=c(n)\in\R^n$ satisfy
\begin{align*}
&\mbox{for every $\kappa>0$ there is $n_0(\kappa)\in\N$
such that for all $n\geq n_0(\kappa)$,}\\
&\mbox{$c=c(n)$ is $(\kappa^{-1}\log(m_n/n)/n,1-\kappa)$--incompressible.}
\end{align*}
Let $z^*=z^*(n)$ be the optimal objective value of random LP \eqref{lp_max}.
Then for every constant $\varepsilon>0$, we have
$$
\Prob\big\{\sqrt{2\log(m_n/n)}\,z^* \geq 1-\varepsilon\big\}
\geq 1-n^{-\omega(1)}.
$$
In particular,
$\liminf\limits_{n\to\infty} \sqrt{2\log(m_n/n)}\,z^* \geq 1$ almost surely.
\end{theorem}

\begin{remark}
Although the assumptions on the cost vectors in Theorem~\ref{th main lower}
may appear excessively technical, we conjecture that they are
optimal in the following sense: whenever there is $\kappa>0$
such that $c(n)$ is $(o(\log(m_n/n)/n),1-\kappa)$--compressible
for infinitely many $n\in\N$, there is $K>0$ (independent of $n$)
and a sequence of random
matrices $A(n)$ with mutually independent centered entries
of unit variances and subgaussian moments bounded by $K$,
such that
$\liminf\limits_{n\to\infty} \sqrt{2\log(m_n/n)}\,z^* < 1$ almost surely.
\end{remark}

\begin{remark}
As a simple sufficient condition on $c(n)$ to satisfy assumptions
of the above theorem, one may require that
$\|c(n)\|_\infty
=o\big(\log^{-1/2}(m_n/n)\big)=0$.
\end{remark}

We refer to Section~\ref{secstructural} for numerical experiments
regarding the assumptions on the cost vectors.

\medskip

As we mentioned in the introduction, 
the main idea of our proof of Theorem~\ref{th main lower}
is to construct a feasible solution $x$ to the LP \eqref{lp_max}
satisfying $\sqrt{2\log(m_n/n)}\,\langle c,x\rangle \geq 1-\varepsilon$
via an iterative process. 
We start by defining a vector $x_0$
as the $(1-\varepsilon)(2\log(m_n/n))^{-1/2}$
multiple of the cost vector $c$. 
Our goal is to add a perturbation to $x_0$ which would restore
its feasibility.
We shall select a vector $x_1\in c^\perp$
which would repair the constraints violated by $x_0$, i.e $\langle \row_i(A),x_0+x_1\rangle\leq 1$
for all $i\in[m_n]$ with $\langle \row_i(A),x_0\rangle> 1$.
The vector $x_1$ may itself introduce some constraints' violation; 
to repair those constraints, we will find another vector $x_2\in c^\perp$,
and consider $x_0+x_1+x_2$, etc. 
At $h$--th iteration, the vector $x_h$ will be chosen
as a linear combination of the matrix rows having large
scalar products with $x_{h-1}$.
The process continues for some finite number of steps $s$
which depends on the values of $m_n$ and $n$. We shall prove that the resulting vector
$x:=x_0+x_1+\dots+x_s$ is feasible with high probability.
To carry this argument over, we start with a number of
preliminary estimates concerning subgaussian
vectors as well as certain estimates on the condition numbers
of random matrices.

\bigskip

We write $[m]$ for the set of integers $\{1,2,\dots,m\}$.
Further, given a finite subset of integers $I$, we write $\R^I$
for the linear real space of $|I|$--dimensional
vectors with components indexed over $I$.
Given a set or an event $S$, we will write $\chi_S$
for the indicator of the set/event.

\subsection{Auxiliary results}

Proposition \ref{alkfjbojhrbljafajkn} and a simple union bound argument yields
\begin{corollary}\label{akgjrbgouerhfbiewufbhifhu}
For every $K\geq 1$ there is a number $K_1=K_1(K)>0$
depending only on $K$ with the following property.
Let $k\leq m/e$, 
let $I\subset[m]$ be a $k$--set, and let $V$
be a random vector in $\R^{[m]\setminus I}$ with mutually
independent centered $K$--subgaussian components of unit variance.
Then the event
\begin{align*}
\bigg\{
&\mbox{The number of indices $i\in [m]\setminus I$ with $V_i\geq
K_1\sqrt{\log (m/k)}$ is at most $k/2$, and }\\
&\Big(\sum_{i\in [m]\setminus I}V_i^2\,\chi_{\{V_i\geq
K_1\sqrt{\log (m/k)}\}}\Big)^{1/2}\leq K_1\sqrt{k\log (m/k)}
\bigg\}
\end{align*}
has probability as least $1-2\big(\frac{k}{m}\big)^{2k}$.
\end{corollary}
\begin{proof}
Set
$$
L:=K_1\sqrt{\log (m/k)},
$$
where $K_1=K_1(K)\geq 1$ is chosen sufficiently large
(the specific requirements on $K_1$ can be extracted from
the argument below).
Since the vector $V$ is subgaussian,
every component of $V$ satisfies
$$
\Prob\big\{V_i\geq L\big\}\leq \exp(-K_2 L^2)
$$
for some $K_2>0$ depending only on $K$.
Hence, the probability that more than $k/2$ components of $V$
are greater than $L$, can be estimated from above by
\begin{equation}\label{argkjbnglkjnfalkfmns}
{m\choose \lceil k/2\rceil}\exp(- K_2 L^2\,\lceil k/2\rceil)
\leq 
\bigg(\frac{em}{\lceil k/2\rceil}\bigg)^{\lceil k/2\rceil}
\exp(-K_2 L^2\,\lceil k/2\rceil).
\end{equation}
If $K_1$ is chosen so that $\frac{em}{\lceil k/2\rceil}
\exp(-K_2 L^2)\leq \big(\frac{k}{m}\big)^4$ then
the last expression in \eqref{argkjbnglkjnfalkfmns} is bounded 
above by
$\big(\frac{k}{m}\big)^{2k}$.

Similarly to the above and assuming $K_1$
is sufficiently large, in view of Proposition~\ref{alkfjbojhrbljafajkn},
for any choice of a $\lfloor k/2\rfloor$--subset
$J\subset[m]\setminus I$, 
the probability that the vector $(V_i)_{i\in J}$
has the Euclidean norm greater than $L\sqrt{k}$, is bounded 
above by
$$
{m\choose \lfloor k/2\rfloor}^{-1}\Big(\frac{k}{m}\Big)^{2k}.
$$
Combining the two observations, we get the result.
\end{proof}

\bigskip

Our next observation concerns random matrices with subgaussian entries.
Recall that the largest and smallest singular
values of an $n\times k$ ($k\leq n$) matrix $M$
are defined by
$$
s_{\max}(M)=\|M\|:=\sup\limits_{y:\,\|y\|_2=1}\|My\|_2,
$$
and
$$
s_{\min}(M):=\inf\limits_{y:\,\|y\|_2=1}\|My\|_2.
$$

\begin{prop}\label{aliuhfpoijfnijn}
For every $K>0$ there are $K_1,K_2>0$ depending only on $K$
with the following property.
Let $\frac{n}{k}\geq K_1$ and let $M$ be an $n\times k$
random matrix with mutually independent centered subgaussian entries
of unit variance and subgaussian moment at most $K$.
Further, let $P$ be a non-random $n\times n$
projection operator of rank $n-1$ i.e $P$ is the orthogonal
projection onto an $(n-1)$--dimensional subspace of $\R^n$.
Then 
$$
\Prob\bigg\{\frac{s_{\max}(PM)}{s_{\min}(PM)}\leq 2\quad
\mbox{and}\quad s_{\min}(PM)\geq \sqrt{n}/2\bigg\}
\geq 1-2\exp(-K_2\,n).
$$
\end{prop}
\begin{proof}
We will assume that the aspect ratio $K_1$
is sufficiently large (the value can be extracted from the proof below).
Let $w$ be the unit vector
in the null space of $P$ (note that it is determined uniquely up to changing the sign).

Fix any unit vector $y\in \R^k$.
Then the random vector $My$ has mutually independent
components of unit variance, and is subgaussian, with
the subgaussian moment of the components depending only on $K$
(see \cite[Section~2.6]{VershyninsBook}).
Applying Proposition~\ref{alkfjbojhrbljafajkn}, we get
\begin{equation}\label{aljghbergojehgbolej}
\Prob\big\{
\big|
\|My\|_2-\sqrt{n}\big|\geq t
\big\}\leq 2\exp(-\tilde K t^2),\quad t>0,
\end{equation}
where $\tilde K>0$ only depends on $K$.
Since the projection operator $P$ acts as a contraction
we have $\|PMy\|_2\leq \|My\|_2$ everywhere on the probability space.
On the other hand, 
$\|PMy\|_2\geq \|My\|_2-|\langle w,My\rangle|$,
where, again according to \cite[Section~2.6]{VershyninsBook},
$\langle w,My\rangle$ is a subgaussian random variable
with the subgaussian moment only depending on $K$, and therefore
\begin{equation}\label{alskfjgnrlgklrekgengl}
    \Prob\big\{
\|My\|_2-\|PMy\|_2\geq t
\big\}\leq 2\exp(-\hat K t^2),\quad t>0.
\end{equation}
Combining \eqref{aljghbergojehgbolej} and \eqref{alskfjgnrlgklrekgengl},
we get
$$
\Prob\big\{
\big|
\|PMy\|_2-\sqrt{n}\big|\geq t
\big\}\leq 2\exp(-K_3 t^2),\quad t>0,
$$
for some $K_3>0$ depending only on $K$.

The rest of the proof of the proposition is based on the standard covering argument
(see, for example, \cite[Chapter~4]{VershyninsBook}),
and we only sketch the idea. 
Consider a {\it Euclidean $\delta$--net $\Net$}
on $S^{k-1}$ (for a sufficiently small constant $\delta\in(0,1/2]$) i.e
a non-random discrete subset of the Euclidean sphere such that for every vector $u$
in $S^{k-1}$ there is some vector in $\Net$ at distance at most $\delta$ from $u$.
It is known that $\Net$ can be chosen to have cardinality at most
$\big(\frac{2}{\delta}\big)^k$ \cite[Section~4.2]{VershyninsBook}.
In view of the last probability estimate, the event
$$
\big\{\|PMy\|_2-\sqrt{n}\big|\leq 0.1\sqrt{n}\mbox{ for all }y\in\Net\big\}
$$
has probability at least $1-2\big(\frac{2}{\delta}\big)^k\exp(-K_3 n/100)$.
Everywhere on that event, we have (see \cite[Section~4.4]{VershyninsBook})
$$
s_{\max}(PM)=\|PM\|\leq\frac{1}{1-\delta}\max\limits_{y\in\Net}\|PMy\|_2
\leq \frac{1.1\sqrt{n}}{1-\delta},
$$
and 
$$
s_{\min}(PM)\geq \min\limits_{y\in\Net}\|PMy\|_2-\delta\|PM\|
\geq 0.9\sqrt{n}-\frac{1.1\delta\sqrt{n}}{1-\delta},
$$
and, in particular,
$$
\frac{s_{\max}(PM)}{s_{\min}(PM)}\leq 2\quad
\mbox{and}\quad s_{\min}(PM)\geq \sqrt{n}/2,
$$
assuming that $\delta$ is a sufficiently small positive constant.
It remains to note that, as long as the aspect ratio $K_1$
is sufficiently large, the quantity
$2\big(\frac{2}{\delta}\big)^k\exp(-K_3 n/100)$
is exponentially small in $n$.
\end{proof}
\begin{remark}
It can be shown that $\frac{s_{\max}(PM)}{s_{\min}(PM)}\leq
1+\varepsilon$ and $s_{\min}(PM)\geq (1-\varepsilon)\sqrt{n}$
with probability exponentially close to one for arbitrary constant $\varepsilon>0$
as long as $K_1=K_1(\varepsilon)$ is sufficiently large.
\end{remark}

\begin{corollary}\label{aklrjgnroigoini}
For every $K>0$ there are $K_1,K_2>0$ depending only on $K$
with the following property.
Let $\frac{n}{k}\geq K_1$, and let 
$P:\R^n\to\R^n$ be a non-random orthogonal
projection of rank $n-1$.
Further, let 
$V_{1},V_{2},\dots,V_{k}$ be independent random vectors in $\R^{n}$
with mutually independent centered $K$--subgaussian components
of unit variance.
Let $U$ be any non-random vector in $\R^k$.
Then with probability $1-\exp(-K_2 n)$
there is the unique vector $y$ in the linear
span of $P(V_{1}),P(V_{2}),\dots,P(V_{k})$ such that
\begin{equation}\label{alfhbrefojlhbvlajhb}
\langle V_{j},y\rangle = U_j,\quad 1\leq j\leq k, 
\end{equation}
and, moreover, that vector satisfies $\|y\|_2\leq 4\|U\|_2/\sqrt{n}$.
\end{corollary}
\begin{proof}
Denote by $M$ the $n\times k$ matrix
with columns $V_{1},V_{2},\dots,V_{k}$, and let
the random vector $v$ in $\R^k$ satisfy $(PM)^\top PMv=U$.
We then set $y:=PMv$. Note that $y$ is 
in the linear
span of $P(V_{1}),P(V_{2}),\dots,P(V_{k})$
and satisfies \eqref{alfhbrefojlhbvlajhb}.
Further, we can write
$$
\|y\|_2\leq s_{\max}(PM)\,\|v\|_2
\leq \frac{s_{\max}(PM)\,\|U\|_2}{s_{\min}((PM)^\top PM)}
= \frac{s_{\max}(PM)\,\|U\|_2}{s_{\min}(PM)^2}.
$$
It remains to apply Proposition~\ref{aliuhfpoijfnijn}
to get the result.
\end{proof}

The next lemma provides a basis for a discretization argument
which we will apply in the proof of the theorem.
\begin{lemma}\label{kgjnovubpogiwpi}
There is a universal constant $K_1\geq 1$ with the following property.
Let $k\geq 1$. Then there is a set $\Net_k$ of vectors
in $\R^k$ satisfying all of the following:
\begin{itemize}
    \item All vectors in $\Net_k$ have non-negative components taking values in
    $\frac{1}{\sqrt{k}}\N_0$;
    \item The Euclidean norm of every vector in $\Net_k$ is in the interval $[1,2]$;
    \item For every unit vector $u\in \R^k$ there is a vector $y=y(u)\in\Net_k$
    such that $u_j\leq y_j$, $j\leq k$;
    \item The size of $\Net_k$ is at most $(K_1)^k$.
\end{itemize}
\end{lemma}
\begin{proof}
Define $\Net_k$ as the collection of all vectors in $\R^k$ with non-negative
components taking values in $\frac{1}{\sqrt{k}}\N_0$ and with the Euclidean
norm in the interval $[1,2]$.
Observe that for every unit vector $u\in \R^k$, by defining $y$ coordinate-wise as
$$
y_j:=\frac{1}{\sqrt{k}}\,\big\lceil\sqrt{k}|u_j|\big\rceil,\quad j\leq k,
$$
we get a vector with Euclidean norm in the interval $[1,2]$,
i.e a vector from $\Net_k$. Thus, $\Net_k$ constructed this way satisfies the first
three properties listed above. It remains to verify the upper bound
on $|\Net_k|$.

To estimate the size of $\Net_k$ we apply the standard volumetric argument.
Observe that $|\Net_k|$ can be bounded from above by the number of integer lattice
points in the rescaled Euclidean ball $2\sqrt{k}B_2^k$ of radius $2\sqrt{k}$.
That, in turn, is bounded above by the maximal number of non-intersecting parallel
translates of the unit cube $(0,1)^k$ which can be
placed inside the ball $3\sqrt{k}B_2^k$. Comparison of the Lebesgue volumes
of $3\sqrt{k}B_2^k$ and $(0,1)^k$ confirms that the latter is
of order $\exp(O(k))$, completing the proof.
\end{proof}

\bigskip

Equipped with Corollaries~\ref{akgjrbgouerhfbiewufbhifhu}
and~\ref{aklrjgnroigoini}, and Lemma~\ref{kgjnovubpogiwpi},
we can complete the proof of the theorem.

\subsection{Proof of Theorem~\ref{th main lower}}


First, suppose that $m_n=n^{\omega(1)}$.
Observe that in this case $\sqrt{2\log (m_n/n)}=
(1-o(1))\sqrt{2\log m_n}$.
Fix a small constant $\varepsilon>0$.
Define $\delta_n:=2(1+\varepsilon/2)(\log m_n)/n$,
so that by the assumptions of the theorem,
for every constant $\kappa\in(0,1)$ and all large enough $n$, the cost vectors 
$c=c(n)$ are  $(\kappa^{-1}\delta_n,1-\kappa)$--incompressible.
Hence, applying Proposition~\ref{aljhofuhwbfohjbfldsajfbalkj},
we get for all large enough $n$ and every $i\leq m_n$,
$$
\Prob\big\{(Ac)_i\geq (1+\varepsilon/2)\sqrt{\delta_n n}\big\}
\leq \exp(-\delta_n n/2),
$$
implying
$$
\Prob\big\{\|Ac\|_\infty\geq (1+\varepsilon/2)\sqrt{\delta_n n}\big\}
\leq m_n\exp(-\delta_n n/2)=n^{-\omega(1)}.
$$
We infer that, by considering
$x=x(n):=\frac{c}{(1+\varepsilon/2)\sqrt{2(1+\varepsilon/2)\log m_n}}$,
the optimal objective value of the LP satisfies
$$
\sqrt{2\log (m_n/n)}\,z^*\geq
\frac{\sqrt{2\log (m_n/n)}}{(1+\varepsilon/2)\sqrt{2(1+\varepsilon/2)\log m_n}}
> 1-\varepsilon
$$
with probability $1-n^{-\omega(1)}$,
and the statement of the theorem follows.

In view of the above, to complete the proof
it is sufficient to verify the statement under the assumption $m_n=n^{O(1)}$.
In what follows, we assume that
$m_n\leq n^{K'}$ for some number $K'>0$ independent of $n$,
and $\varepsilon>0$ is a small constant.

\bigskip

We shall construct a ``good'' event of high probability, condition on a realization
of the matrix $A=A(n)$ from that event, and then follow the iterative scheme
outlined at the beginning of the section to construct a feasible solution.
Set
$$
x_0=x_0(n):=\frac{(1-\varepsilon)c}{\sqrt{2\log (m_n/n)}},
$$
and define an integer $k_0=k_0(n)$ as
$$
k_0:=\big\lceil 
m_n(n/m_n)^{1+\varepsilon/4}
\big\rceil=\big\lceil n(n/m_n)^{\varepsilon/4}\big\rceil.
$$

\bigskip

{\bf{}Definition of a ``good'' event.}
The event to be conditioned on is constructed as an intersection of multiple events.
We start by defining
\begin{align*}
\Event_c:=\big\{
&\mbox{At most $k_0$ components of the vector $Ax_0$ are greater than
$1-\varepsilon/2$,}\\
&\mbox{and the sum of squares of those components
is bounded above by $\tilde K^2\,k_0$}
\big\},
\end{align*}
where $\tilde K>0$ is a sufficiently large universal constant
whose value can be extracted from the argument below.
Observe that, by our definition of $k_0$
and by our assumption that $\varepsilon$ is small,
for all large enough $n$, 
$$
(1+\varepsilon/16)
\sqrt{2(1+\varepsilon/16)\log (m_n/k_0)}\leq \frac{1-\varepsilon/2}{1-\varepsilon}
\sqrt{2\log (m_n/n)},
$$
whence for every $i\leq m_n$, in view of the definition of $x_0$,
\begin{equation}\label{eq:oajnkjnflakjnflak}
\Prob\big\{(Ax_0)_i\geq 1-\varepsilon/2\big\}
\leq \Prob\big\{(Ac)_i\geq (1+\varepsilon/16)
\sqrt{2(1+\varepsilon/16)\log(m_n/k_0)}\big\}.
\end{equation}
Define $\delta_n:=2(1+\varepsilon/16)\log(m_n/k_0)/n$,
and observe that, by the assumptions of the theorem,
for every constant $\kappa\in(0,1)$ and all large enough $n$, the cost vectors 
$c(n)$ are  $(\kappa^{-1}\delta_n,1-\kappa)$--incompressible.
Hence, we can apply Proposition~\ref{aljhofuhwbfohjbfldsajfbalkj}
to get for all large enough $n$,
$$
\Prob\big\{(Ac)_i\geq (1+\varepsilon/16)\sqrt{\delta_n n}\big\}
\leq \exp\big(-\delta_n n/2\big).
$$
In particular, in view of \eqref{eq:oajnkjnflakjnflak}, the probability that $(Ax_0)_i\geq 1-\varepsilon/2$ for more than
$k_0$ indices $i$, is bounded above by
$$
{m_n\choose k_0}\exp\big(-\delta_n n k_0/2\big)
\leq \bigg(\frac{em_n\exp(-\delta_n n/2)}{k_0}\bigg)^{k_0}
=
\bigg(e\,\Big(\frac{k_0}{m_n}\Big)^{\varepsilon/16}\bigg)^{k_0}.
$$
On the other hand, for every $t\geq \frac{2\sqrt{k_0}}{\sqrt{2\log(m_n/n)}}$,
by Proposition~\ref{alkfjbojhrbljafajkn}
and by the union bound over $k_0$--subsets of $[m_n]$,
the probability that the sum of squares of largest $k_0$ components
of $Ax_0$ is greater than $t^2$, is bounded above by
$$
{m_n\choose k_0}\cdot 2\exp(-K_3 t^2\log(m_n/n))
\leq 2\,\Big(\frac{em_n}{k_0}\Big)^{k_0}\exp(-K_3 t^2\log(m_n/n)),
$$
where $K_3>0$ is a universal constant. Taking $t$
to be a sufficiently large constant multiple of $\sqrt{k_0}$ and
combining the above estimates, we obtain that the event
$\Event_c$ has probability $1-n^{-\omega(1)}$.

\bigskip

Fix for a moment any subset $I\subset[m_n]$ of size at most $k_0$
and at least $\log n$. Further, denote by $\Net_I$
the non-random discrete subset of vectors in $\R^I$ with the properties
listed in Lemma~\ref{kgjnovubpogiwpi}.
Denote by $P=P(n):\R^n\to\R^n$ the projection onto the orthogonal complement of
the cost vector $c=c(n)$.
For every $u\in \Net_I$, define
\begin{align*}
\Event_{I,u}
:=\bigg\{
&\mbox{There is a unique vector $y$
in the linear
span of $P(\row_i(A))$, $i\in I$,}\\
&\mbox{such that
$\langle y,\row_i(A)\rangle=u_i$, $i\in I$; moreover, that vector satisfies:}\\
&\mbox{(a) $\|y\|_2\leq 8/\sqrt{n}$;}\\
&\mbox{(b) number of indices $i\in [m_n]\setminus I$ with $(Ay)_i\geq
K''\sqrt{\log (m_n/|I|)}/\sqrt{n}$ is at most $|I|/2$;}\\
&(c)\;\Big(\sum_{i\in [m_n]\setminus I} (Ay)_i^2\,\chi_{\{(Ay)_i\geq
K''\sqrt{\log (m_n/|I|)}/\sqrt{n}\}}\Big)^{1/2}
\leq K''\sqrt{|I|\log (m_n/|I|)}/\sqrt{n}
\bigg\},
\end{align*}
where $K''=K''(K)>0$ is a sufficiently large constant depending only on $K$.
We claim that, assuming $n$ is sufficiently large, the event $\Event_{I,u}$
has probability close to one. Indeed, since the collections of random vectors
$\{\row_i(A)\}_{i\in I}$ and $\{\row_i(A)\}_{i\in [m_n]\setminus I}$ are independent,
we have
\begin{align*}
\Prob(\Event_{I,u})
\geq \Prob\bigg\{
&\mbox{There is a unique $y\in{\rm span}\,\{P(\row_i(A)),\,i\in I\}$}\\
&\mbox{such that
$\langle y,\row_i(A)\rangle=u_i$, $i\in I$; and $\|y\|_2\leq 8/\sqrt{n}$}\bigg\}\cdot
\inf\limits_y\Prob\big\{\mbox{$y$ satisfies (b) and (c)}\big\},
\end{align*}
where the infimum is taken over all non-random vectors $y\in\R^n$
satisfying $\|y\|_2\leq 8/\sqrt{n}$.
Using that $\|u\|_2\leq 2$
for every $u\in \Net_I$ and 
applying Corollary~\ref{aklrjgnroigoini} to the first probability on the right-hand side,
we get for all large enough $n$
$$
\Prob(\Event_{I,u})
\geq
\big(1-\exp(-\Omega(n))\big)\;
\inf\limits_y\Prob\big\{\mbox{$y$ satisfies (b) and (c)}\big\}.
$$
Further, applying Corollary~\ref{akgjrbgouerhfbiewufbhifhu}
to the second probability, we obtain that the event
$\Event_{I,u}$ has probability at least 
$$
\big(1-\exp(-\Omega(n))\big)\;\big(1-2(|I|/m_n)^{2|I|}\big)
\geq
1-2(|I|/m_n)^{2|I|}-\exp(-\Omega(n)),$$
whence the intersection of events $\bigcap\limits_{I\in \mathcal I}\bigcap\limits_{u\in\Net_I}\Event_{I,u}$
taken over the collection $\mathcal I$ of all
subsets of $[m_n]$ of size at least $\log n$
and at most $k_0$, has probability $1-n^{-\omega(1)}$.

\bigskip

Define the ``good'' event
$$
\Event_{good}:=\Event_c\;\cap\;\bigcap\limits_{I\in \mathcal I}\bigcap\limits_{u\in\Net_I}\Event_{I,u},
$$
so that $\Prob(\Event_{good})=1-n^{-\omega(1)}$.
Our goal is to show that, conditioned on any realization of $A$
from $\Event_{good}$, the optimal objective value of the LP
is at least $\frac{1-\varepsilon}{\sqrt{2\log (m_n/n)}}$. From now on, we assume that the matrix $A$
is non-random and satisfies the conditions from the definition of $\Event_{good}$.
For reader's convenience, we outline an algorithm for constructing a feasible solution
$x$ certifying the aforementioned lower bound on $z^*$:

\bigskip

\begin{algorithm}[H]
    \SetAlgoLined
    \KwIn{ $A$: ${m_n \times n}$ non-random matrix satisfying all the conditions defined in $\Event_{good}$.} 
    \KwOut{ $x$: feasible solution to LP \eqref{lp_max} satisfying $\langle x, c \rangle \geq  \frac{1-\varepsilon}{\sqrt{2\log (m_n/n)}}$.}
    $k_0 \leftarrow \big\lceil m_n(n/m_n)^{1+\varepsilon/4} \big\rceil$;\\
    $x_0 \leftarrow \frac{(1-\varepsilon) c}{\sqrt{2\log (m_n/n)}}$;\\
    $I_0 \leftarrow$  $k_0$--subset of $[m_n]$ corresponding to $k_0$ largest components of the vector $Ax_0$;\\
    $z_0\in\R^{I_0} \leftarrow$ the restriction of the vector $Ax_0$ to $\R^{I_0}$;\\
    $u_0\in\Net_{I_0} \leftarrow$  a vector which majorizes $z_0/\|z_0\|_2$ coordinate-wise; ($\Net_{I_0}$ is the discrete subset of $\R^{I_0}$ in  Lemma \ref{kgjnovubpogiwpi}) \\
    $y_1 \leftarrow$ the unique vector in the linear combination of $P(\row_i(A))$, $i\in I_0$ ($y_1$ orthogonal to $c$) such that
    $\langle y_1,\row_i(A)\rangle=(u_0)_i, i\in I_0$; ($P$ is orthogonal projection to $c^{\perp}$)\\
    $x_1 \leftarrow -\|z_0\|_2\,y_1$;\\
    $s \leftarrow \lceil \log_2 k_0-\frac{1}{4}\log_2 n\rceil$;\\
    \For{$1 \leq h\leq s-1$ }{ 
    $k_{h} \leftarrow \lfloor k_{h-1}/2\rfloor$;\\
    $I_h \leftarrow $  $k_h$--subset of $[m_n]\setminus I_{h-1}$ corresponding to $k_h$ largest components of $Ax_h$;\\
    $z_{h}\in\R^{I_{h}} \leftarrow$ the restriction of the vector $Ax_{h}$ to $\R^{I_{h}}$;\\
    $u_{h}\in\Net_{I_{h}} \leftarrow$  a vector which majorizes $z_{h}/\|z_{h}\|_2$ coordinate-wise;\\ 
    $y_{h+1} \leftarrow$ the unique vector in the linear span of $P(\row_i(A))$, $i\in I_h$,
    such that $\langle \row_i(A), y_{h+1}\rangle = (u_h)_i, i\in I_h$;\\
    $x_{h+1} \leftarrow -\|z_{h}\|_2\,y_{h+1} $;\\
    }
    $x \leftarrow x_0+x_1+\dots+x_s $;\\
    return $x$;\\
    \caption{Finding a feasible solution to LP \eqref{lp_max}}
    \label{alg proof 3.1}
\end{algorithm}

\bigskip

Having outlined the algorithm, we proceed with a formal verification of its correctness.

{\bf{}Definition of $x_1$.}
Let $I_0$ be the $k_0$--subset of $[m_n]$
corresponding to $k_0$ largest components of the vector $Ax_0$
(we resolve ties arbitrarily).
In view of our conditioning on $\Event_c$, for every $i\in[m_n]\setminus I_0$
we have
$
\langle \row_i(A),x_0\rangle \leq 1-\varepsilon/2.
$
Denote by $z_0\in\R^{I_0}$ the restriction of the vector $Ax_0$
to $\R^{I_0}$, and let $u_0\in\Net_{I_0}$ be
a vector which majorizes $z_0/\|z_0\|_2$ coordinate-wise.
Note that in view of the definition of $\Event_c$, the Euclidean norm of $z_0$
is of order $O(\sqrt{k_0})$.
Further, since we conditoned on $\Event_{I_0,u_0}$, we find a unique vector $y_1$ in the linear combination
of $P(\row_i(A))$, $i\in I_0$ (and, in particular, $y_1$ orthogonal to $c$) such that
$$\langle y_1,\row_i(A)\rangle=(u_0)_i,\quad i\in I_0,$$ 
satisfying the conditions
(a), (b), (c) from the definition of $\Event_{I_0,u_0}$.
Define 
$$x_1:=-\|z_0\|_2\,y_1,$$ 
so that
$\langle \row_i(A),x_1\rangle=-\|z_0\|_2\,(u_0)_i$, $i\in I_0$.
Observe that with such a definition of $x_1$, the sum $x_0+x_1$ satisfies
\begin{equation}
\label{eq: aljhbkajhfbjqhfblajhbfksjah}
\langle \row_i(A),x_0+x_1\rangle
=(z_0)_i-\|z_0\|_2\,(u_0)_i
\leq 0,\quad i\in I_0,
\end{equation}
whereas 
$$\langle x_0+x_1,c\rangle =\langle x_0,c\rangle =
\frac{1-\varepsilon}{\sqrt{2\log(m_n/n)}}.$$
Applying the definition of the event $\Event_{I_0,u_0}$
and the bound $\|z_0\|_2=O(\sqrt{k_0})$, we get that
the vector $x_1$ satisfies
\begin{align*}
&\mbox{The number of indices $i\in [m_n]\setminus I_0$ with
$(Ax_1)_i\geq \tilde K\sqrt{(k_0/n)\log (m_n/k_0)}$ is at most $k_0/2$;}\\
&\Big(\sum_{i\in [m_n]\setminus I_0} (Ax_1)_i^2\,\chi_{\{(Ax_1)_i\geq
\tilde K\sqrt{(k_0/n)\log (m_n/k_0)}\}}\Big)^{1/2}
\leq \tilde K\sqrt{(k_0^2/n)\log (m_n/k_0)},
\end{align*}
where $\tilde K=\tilde(K)>0$ depends only on $K$.
We will simplify the conditions on $x_1$. First, observe that
in view of the definition of $k_0$ and since $m_n=n^{O(1)}$, we can assume that
$\tilde K\sqrt{(k_0/n)\log (m_n/k_0)}\leq\varepsilon/16$.
Further, we can estimate $\tilde K\sqrt{(k_0^2/n)\log (m_n/k_0)}$
from above by $\sqrt{\lfloor k_{0}/2\rfloor}$. Thus, the vector $x_1$ satisfies
\begin{align}
&\mbox{The number of indices $i\in [m_n]\setminus I_0$ with
$(Ax_1)_i\geq \varepsilon/16$ is at most $\lfloor k_{0}/2\rfloor$;}\nonumber\\
&\Big(\sum_{i\in [m_n]\setminus I_0} (Ax_1)_i^2\,\chi_{\{(Ax_1)_i\geq
\varepsilon/16\}}\Big)^{1/2}
\leq \sqrt{\lfloor k_{0}/2\rfloor}.\label{eq:lkjbljhfbwjfhbqwjf}
\end{align}

\bigskip

{\bf{}Definition of $x_2,x_3,\dots$}
The vectors $x_2,x_3,\dots$ are constructed via an inductive process.
Let 
$$s:=\lceil \log_2 k_0-\frac{1}{4}\log_2 n\rceil.$$
We define numbers 
$k_1,k_2,\dots,k_{s}$ recursively by the formula
$$
k_{h}:=\lfloor k_{h-1}/2\rfloor,\quad 1\leq h\leq s.
$$
With $n$ large, we then have
$$
k_{s}=n^{1/4\pm o(1)}.
$$
Further, assuming that for some $1\leq h\leq s-1$ the vector $x_h$ has been defined, we let $I_h$
be a $k_h$--subset of $[m_n]\setminus I_{h-1}$ corresponding to $k_h$ largest 
components of $Ax_h$ (with ties resolved arbitrarily).

\medskip

Let $1\leq h\leq s-1$, and assume that a vector $x_{h}$
has been constructed that satisfies
\begin{align*}
&\mbox{The number of indices $i\in [m_n]\setminus I_{h-1}$ with
$(Ax_h)_i\geq (3/4)^{h-1}\varepsilon/16$ is at most $k_h$;}\\
&\Big(\sum_{i\in [m_n]\setminus I_{h-1}} (Ax_h)_i^2\,\chi_{\{(Ax_h)_i\geq
(3/4)^{h-1}\varepsilon/16\}}\Big)^{1/2}
\leq \sqrt{k_{h}}
\end{align*}
(observe that $x_1$ satisfies the above assumptions, in view of \eqref{eq:lkjbljhfbwjfhbqwjf}, which
provides a basis of induction).
We denote by $z_{h}\in\R^{I_{h}}$
the restriction of the vector $Ax_{h}$
to $\R^{I_{h}}$ (note that $z_h$ has the Euclidean norm
at most $\sqrt{k_{h}}$ by the induction hypothesis),
and let $u_{h}\in\Net_{I_{h}}$ be
a vector which majorizes $z_{h}/\|z_{h}\|_2$ coordinate-wise.
Similarly to the above, we let 
$y_{h+1}$ be the unique
vector in the linear span of $P(\row_i(A))$, $i\in I_h$,
such that 
$$\langle \row_i(A), y_{h+1}\rangle = (u_h)_i,\quad
i\in I_h,$$
and satisfying the conditions (a), (b), (c)
from the definition of $\Event_{I_h,u_h}$,
and define 
\begin{equation}\label{eq: alkdhjbfkajhfbksjfhbskj}
x_{h+1}:=-\|z_{h}\|_2\,y_{h+1},      
\end{equation}
so that
\begin{equation}\label{eq: alkjbljhfbofhubfojhbaljf}
\langle \row_i(A), x_{h+1}\rangle = -\|z_{h}\|_2\,(u_h)_i,\quad i\in I_h.
\end{equation}
By the definition of $\Event_{I_h,u_h}$,
the vector $Ax_{h+1}$ satisfies
\begin{equation*}
\begin{split}
&\mbox{Number of indices $i\in [m_n]\setminus I_h$ with $(A x_{h+1})_i\geq
\tilde K\sqrt{k_{h}\log (m_n/k_h)}/\sqrt{n}$ is at most $k_{h+1}$;}\\
&\Big(\sum_{i\in [m_n]\setminus I_h} (Ax_{h+1})_i^2\,\chi_{\{(Ax_{h+1})_i\geq
K''\sqrt{k_{h}\log (m_n/k_h)}/\sqrt{n}\}}\Big)^{1/2}
\leq \tilde K\sqrt{k_{h}^2\log (m_n/k_h)}/\sqrt{n},
\end{split}
\end{equation*}
for some constant $\tilde K=\tilde K(K)>0$.
In is not difficult to check that, with our definition of $k_1,k_2,\dots$
and the assumption $m_n=n^{O(1)}$
and as long as $n$ is large,
the quantity $\tilde K\sqrt{k_{h}\log (m_n/k_h)}/\sqrt{n}$
is (much) less than $(3/4)^{h}\varepsilon/16$; further,
$\tilde K\sqrt{k_{h}^2\log (m_n/k_h)}/\sqrt{n}$ is dominated by $\sqrt{k_{h+1}}$.
Thus, $x_{h+1}$ satisfies
\begin{align}
&\mbox{The number of indices $i\in [m_n]\setminus I_h$ with $(A x_{h+1})_i\geq
(3/4)^{h}\varepsilon/16$ is at most $k_{h+1}$;}\label{eq:lakjenflkjfnlkajfnals}\\
&\Big(\sum_{i\in [m]\setminus I_h} (Ax_{h+1})_i^2\,\chi_{\{(Ax_{h+1})_i\geq
(3/4)^{h}\varepsilon/16\}}\Big)^{1/2}
\leq 
\tilde K\sqrt{k_{h}^2\log (m_n/k_h)}/\sqrt{n}
\leq \sqrt{k_{h+1}},\label{apijnogijnrgoi}
\end{align}
completing the induction step.

\smallskip

Next, we gather the properties of the constructed sets $I_h$ and vectors $x_h$
that will be useful to us.
Observe that, in view of \eqref{eq: alkjbljhfbofhubfojhbaljf} and the
condition that $u_{h}\in\Net_{I_{h}}$ majorizes $z_{h}/\|z_{h}\|_2$ coordinate-wise
(see also \eqref{eq: aljhbkajhfbjqhfblajhbfksjah}), we have
\begin{equation}
\label{eq: lkajfh98752987}
\langle \row_i(A),x_h+x_{h+1}\rangle\leq 0, \quad 0\leq h\leq s-1,\quad i\in I_h
\end{equation}
(implying that constraints violated by $x_h$ are repaired by $x_{h+1}$)
and, as a weaker property that follows directly from \eqref{eq: alkdhjbfkajhfbksjfhbskj},
\begin{equation}
\label{eq: aljkhebfkjhfbkjfhbaksj}
\langle \row_i(A),x_{h+1}\rangle\leq 0,\quad
0\leq h\leq s-1,\quad i\in I_h.
\end{equation}
Further, since by the construction for every $1\leq h\leq s$,
the number of indices $i\in [m_n]\setminus I_{h-1}$ with
$(Ax_h)_i\geq (3/4)^{h-1}\varepsilon/16$ is at most $k_h$ (see \eqref{eq:lakjenflkjfnlkajfnals}), we infer that
$$
\langle \row_i(A),x_{h}\rangle< (3/4)^{h-1}\varepsilon/16,\quad 1\leq h\leq s-1,\quad i\in [m_n]\setminus I_{h}.
$$
Finally, the condition \eqref{apijnogijnrgoi} for $h+1=s$
reads
\begin{align*}
\Big(\sum_{i\in [m_n]\setminus I_{s-1}} (Ax_{s})_i^2\,\chi_{\{(Ax_{s})_i\geq
(3/4)^{s-1}\varepsilon/16\}}\Big)^{1/2}
\leq
\tilde K\sqrt{k_{s-1}^2\log (m_n/k_{s-1})}/\sqrt{n}=o(1),
\end{align*}
so that, in particular, 
\begin{equation}
\label{eq: aljhsbfkajshbfla3452jkfbls}
\langle \row_i(A),x_{s}\rangle\leq \varepsilon/16,\quad i\in[m_n]\setminus I_{s-1}.    
\end{equation}

\bigskip

{\bf{}Restoring feasibility.}
The above process produced vectors $x_0,x_1,\dots,x_s$.
Define the sum
$$
x:=x_0+x_1+\dots+x_s.
$$
By the construction, every $x_h$, $1\leq h\leq s$,
lies in the orthogonal complement of $c$,
and hence
$$\langle x, c\rangle=\langle x_0, c\rangle
=
\frac{1-\varepsilon}{\sqrt{2\log(m_n/n)}}.$$

\bigskip

To complete the proof that $z^*\geq \frac{1-\varepsilon}{\sqrt{2\log(m_n/n)}}$, it remains to verify that $x$ is a feasible solution to the linear program.
For any index $i\in[m_n]$,
let 
\begin{itemize}
\item $\chi_0(i)$ be the Boolean indicator of the expression
``$\langle \row_i(A),x_0\rangle\geq 1-\varepsilon/2$'';
\item for each $1\leq h\leq s-1$, let $\chi_h(i)$ be the indicator of 
``$\langle \row_i(A),x_h\rangle\geq (3/4)^{h-1}\varepsilon/16$'';
\item set $\chi_s(i):=0$.
\end{itemize}

Note that, in view of the definition of the sets $I_0,I_1,\dots$,
for every $0\leq h\leq s-1$ we have $\chi_h(i)=1$ {\it only if} $\chi_{h+1}(i)=0$.
Indeed, the condition $\chi_h(i)=1$ implies that $i\in I_h$,
so that \eqref{eq: aljkhebfkjhfbkjfhbaksj} yields $\chi_{h+1}(i)=0$.

\smallskip

From now on, we fix $i\in[m_n]$.
The set $\{0,1,\dots,s\}$
can be partitioned into consecutive subsets $S_1<S_2<\dots<S_p$
in such a way that
\begin{itemize}
    \item The size of each set $S_r$ is either $1$ or $2$;
    \item For each singleton $S_r=\{h\}$, $\chi_h(i)=0$;
    \item For each subset $S_r=\{h,h+1\}$ of size two, $\chi_h(i)=1$ and $\chi_{h+1}(i)=0$.
\end{itemize}
We have
$$
\langle \row_i(A),x\rangle=\sum_{r=1}^p \sum_{h\in S_r}
\langle \row_i(A),x_h\rangle.
$$
Fix for a moment any $r\leq p$ and consider
the corresponding sum $\sum_{h\in S_r}
\langle \row_i(A),x_h\rangle$. 
First, assume that $S_r$ is a singleton; $S_r=\{h\}$.
By our construction of the partition, we have $\chi_h(i)=0$.
\begin{itemize}
\item If $h=0$ then
$$
\langle \row_i(A),x_h\rangle\leq 1-\varepsilon/2;
$$
\item If $1\leq h\leq s-1$ then
$$
\langle \row_i(A),x_h\rangle\leq (3/4)^{h-1}\varepsilon/16;
$$
\item If $h=s$ then, in view of \eqref{eq: aljkhebfkjhfbkjfhbaksj} and \eqref{eq: aljhsbfkajshbfla3452jkfbls}
$$
\langle \row_i(A),x_h\rangle\leq \varepsilon/16.
$$
\end{itemize}
Next, assume that the size of $S_r$ is two; $S_r=\{h,h+1\}$.
Since in this case $\chi_h(i)=1$, we have $i\in I_h$ and, by \eqref{eq: lkajfh98752987},
$$
\langle \row_i(A),x_h+x_{h+1}\rangle\leq 0.
$$

\medskip

Combining the two possible cases and summing over $r$, we get
$$
\langle \row_i(A),x\rangle
\leq 1-\varepsilon/2
+\sum_{h=1}^\infty (3/4)^{h-1}\varepsilon/16
+\varepsilon/16\leq 1,
$$
and the proof is finished.

\subsection{Applications of Theorem~\ref{th main lower}}\label{structcostv}

As a simple sufficient condition on the cost vectors 
which implies the assumptions in Theorem~\ref{th main lower},
we consider a bound on the $\|\cdot\|_\infty$--norms
which, in particular, implies the lower bound in the main Theorem~\ref{th main nongauss}:

\begin{corollary}[Lower bound on $z^*$ under sup-norm delocalization
of the cost vectors]\label{linfty norm cor}
Let $m_n$ satisfy $\lim\limits_{n\to\infty}\frac{m_n}{n}=\infty$,
and let matrices $A=A(n)$ be as in Theorem~\ref{th main lower}.
Consider a sequence of non-random cost vectors $c=c(n)$
satisfying $\lim_{n \rightarrow \infty}\sqrt{\log (m_n/n)}\,\|c\|_\infty=0$.
Then for any constant $\varepsilon>0$ and all sufficiently
large $n$,
$$
\Prob\big\{\sqrt{2\log(m_n/n)}\,z^* \geq 1-\varepsilon\big\}
\geq 1-n^{-\omega(1)},
$$
implying that
$\liminf\limits_{n\to\infty} \sqrt{2\log(m_n/n)}\,z^* \geq 1$ almost surely.
\end{corollary}
\begin{proof}
Fix any constant $\kappa\in(0,1)$, and let $I=I(n)$ be any subset of $[n]$
of size at most $\kappa^{-1}\log(m_n/n)$.
In view of the assumptions on $c$, we have
$$
\sum_{i\in I}c_i^2\leq \|c\|_\infty^2|I|\leq
\kappa^{-1}\log(m_n/n)\|c\|_\infty^2\longrightarrow 0.
$$
Since the choice of the subsets $I(n)$ was arbitrary, the last relation
implies that the vector $c$ is $(\kappa^{-1}\log(m_n/n),1-\kappa)$--incompressible
for all large $n$. It remains to apply Theorem~\ref{th main lower}.
\end{proof}

The next corollary constitutes the lower bound in Corollary~\ref{mw cor nongauss}:
\begin{corollary}[The mean width: a sharp lower bound]\label{mw cor lower}
Let $m_n$ and $A=A(n)$ be as in Theorem~\ref{th main lower},
and assume additionally that $m_n=\exp(o(n))$.
Then the spherical mean width
$\cw(P)$ of the polyhedron $P=P(n)=\{x\in\R^n:\;Ax \leq \mathbf{1}\}$
satisfies
$$
\liminf\limits_{n\to\infty} \sqrt{2\log(m_n/n)}\,\cw(P) \geq 2\quad\mbox{almost surely.}
$$
\end{corollary}
\begin{proof}
In this proof, we assume that for each $n$,
$c=c(n)$ is a uniform random unit vector
in $\R^n$, and that $c(n)$, $n\in \N$, are mutually independent
and independent from $\{A(n),\;n\in\N\}$.
Recall that the mean width of the polyhedron $P=P(n)=\{x\in\R^n:\;Ax\leq {\bf 1}\}$
is defined as
$$
\cw(P)=2\,\Exp_c\,\max\{\langle c,x\rangle:\;x\in P\},
$$
where $\Exp_c$ denotes the expectation taken with respect to $c=c(n)$.
It will be most convenient for us to prove the statement by contradiction.
Namely, assume that there is a small positive number $\delta>0$
such that
$$
\Prob\big\{\liminf\limits_{n\to\infty} \sqrt{2\log(m_n/n)}\cw(P(n))\leq 2-\delta\big\}
\geq \delta.
$$
Denote by $\mathcal F$ the $\sigma$--field
generated by the random matrices $\{A(n),\;n\in\N\}$.
Observe that the event
$$
\Event:=\big\{\liminf\limits_{n\to\infty} \sqrt{2\log(m_n/n)}\cw(P(n))\leq 2-\delta\big\}
$$
is $\mathcal F$--measurable. Condition for a moment on any
realization of the polyhedra $P(n)$, $n\in \N$, from $\Event$,
and let $(n_k)_{k=1}^\infty$ be an increasing sequence of integers
such that $\sqrt{2\log(m_{n_k}/n_k)}\cw(P(n_k))\leq 2-\delta/2$ for every $k$.
The condition
$
\cw(P(n_k))\leq
\frac{2-\delta/2}{\sqrt{2\log(m_{n_k}/n_k)}}
$
implies
\begin{align*}
\frac{1-\delta/4}{\sqrt{2\log(m_{n_k}/n_k)}}&\geq
\Exp_c\,\max\{\langle c(n_k),x\rangle:\;x\in P(n_k)\}\\
&\geq \frac{1-\delta/8}{\sqrt{2\log(m_{n_k}/n_k)}}\,\Prob_c\bigg\{
\max\{\langle c(n_k),x\rangle:\;x\in P(n_k)\}\geq \frac{1-\delta/8}{\sqrt{2\log(m_{n_k}/n_k)}}
\bigg\},
\end{align*}
and hence
$$
\Prob_c\bigg\{
\max\{\langle c(n_k),x\rangle:\;x\in P(n_k)\}< \frac{1-\delta/8}{\sqrt{2\log(m_{n_k}/n_k)}}
\bigg\}
\geq 1-\frac{1-\delta/4}{1-\delta/8},
$$
where $\Prob_c$ is the conditional probability given the realization of $P(n)$, $n\in\N$,
from $\Event$.
By the Borel--Cantelli lemma, the last assertion yields
$$
\Prob_c\bigg\{
\liminf\limits_{n\to\infty}
\sqrt{2\log(m_n/n)}\max\{\langle c(n),x\rangle:\;x\in P(n)\}\leq 1-\delta/8
\bigg\}=1.
$$
Removing the conditioning on $\Event$, we arrive at the estimate
\begin{equation}\label{kgjnerlkejnlfkjnfalkajnf}
\Prob\bigg\{
\liminf\limits_{n\to\infty}
\sqrt{2\log(m_n/n)}\max\{\langle c(n),x\rangle:\;x\in P(n)\}\leq 1-\delta/8
\bigg\}\geq \delta.
\end{equation}

\bigskip

In the second part of the proof, we will show that
the assertion \eqref{kgjnerlkejnlfkjnfalkajnf} leads to contradiction.
Fix for a moment any $\kappa\in(0,1)$.
We claim that $c(n)$
is $(\kappa^{-1}\log(m_n/n),1-\kappa)$--incompressible for 
every large enough $n$ with probability $1-n^{-\omega(1)}$.
Indeed, the standard concentration inequality on the sphere
(see, in particular, \cite[Chapter~5]{VershyninsBook}) implies that
for every choice of a non-random $\lfloor\kappa^{-1}\log(m_n/n)\rfloor$--subset $I\subset[n]$,
$$
\Prob\bigg\{\sqrt{\sum_{i\in I}c_i^2}\geq \sqrt{\frac{K_1|I|}{n}}
+t\bigg\}\leq 2\exp(-K_2 n t^2),\quad t>0,
$$
for some universal constants $K_1,K_2>0$.
Taking the union bound, we then get
\begin{align*}
\Prob&\bigg\{\sqrt{\sum_{i\in I}c_i^2}\geq \sqrt{\frac{K_1|I|}{n}}
+t\;\;\mbox{for some $\lfloor\kappa^{-1}\log(m_n/n)\rfloor$--subset $I$}\bigg\}\\
&\leq 2{n\choose \lfloor\kappa^{-1}\log(m_n/n)\rfloor}\exp(-K_2 n t^2),\quad t>0,
\end{align*}
Our assumption $m_n=\exp(o(n))$ implies that
$$
\sqrt{\frac{K_1\,\lfloor\kappa^{-1}\log(m_n/n)\rfloor}{n}}\leq\kappa/2
\quad \mbox{for all large enough $n$.}
$$
For all such $n$, the last estimate 
with $t:=\kappa/2$ yields
\begin{equation}\label{apirwugokjndakfajnf}
\Prob\bigg\{\sqrt{\sum_{i\in I}c_i^2}\geq \kappa\;\;\mbox{for a $\lfloor\kappa^{-1}\log(m_n/n)\rfloor$--subset $I$}\bigg\}
\leq 2{n\choose \lfloor\kappa^{-1}\log(m_n/n)\rfloor}\exp(-K_2 n \kappa^2/4).
\end{equation}
It remains to observe that as long as $m_n=\exp(o(n))$,
we have
$$
{n\choose \lfloor\kappa^{-1}\log(m_n/n)\rfloor}=\exp(o(n)),
$$
to that the expression on the right hand side of \eqref{apirwugokjndakfajnf}
is bounded above by $2\exp(-\Omega(n))$. This verifies the claim.

\bigskip

With the claim above confirmed, it remains to apply Theorem~\ref{th main lower}.
From the theorem, we obtain that
$$
\liminf\limits_{n\to\infty}\sqrt{2\log(m_n/n)}\max\{\langle c,x\rangle:\;
x\in P(n)\}\geq 1
$$
with probability one. However, that contradicts
\eqref{kgjnerlkejnlfkjnfalkajnf}, completing the proof.
\end{proof}

\section{The upper bound on $z^*$ in subgaussian setting}\label{akjshbfakjhfbkqjwhb}

The goal of this section is to prove the following result:
\begin{theorem}[Upper bound on $z^*$]\label{th main upper}
Fix any $K\geq 1$.
Let $n\to\infty$, $\frac{m_n}{n}\to\infty$.
For each $n$ let $c=c(n)$
be a non-random unit vector in $\R^n$,
and assume
that $\lim\limits_{n\to\infty}\big(\log^{3/2}(m_n/n)\,\|c\|_\infty\big)=0$.
For each
$n$, let $A=A(n)$
be an $m_n\times n$ matrix with
mutually independent centered $K$--subgaussian entries of
unit variances.
Then for every $\varepsilon>0$ and all large $n$ we have
$$
\Prob\bigg\{\sqrt{2\log (m_n/n)}z^*\leq 1+\varepsilon\bigg\}\geq 1- n^{-\omega(1)}.
$$
In particular, $\limsup\limits_{n\to\infty}
\sqrt{2\log(m_n/n)} z^*\leq 1$ almost surely.
\end{theorem}
Observe that the statement of the theorem is equivalent to
the assertion that the intersection of the feasible polyhedron
$P=\{x\in\R^n:\;Ax\leq{\bf 1}\}$ with the affine hyperplane
$$
H:=\bigg\{w+\frac{1+\varepsilon}{\sqrt{2\log (m_n/n)}}\,c:\;w\in c^\perp\bigg\}
$$
is empty with probability $1-n^{-\omega(1)}$.
To prove the result, we will combine the
moderate deviations estimate from Proposition~\ref{aksjfnoifurhfoioin}
with a covering argument. 

\subsection{Auxiliary results}

Proposition~\ref{aksjfnoifurhfoioin} cannot be directly applied
in our setting since the vectors in $H$ generally do not satisfy (even if normalized)
the required $\|\cdot\|_\infty$--norm bound.
To overcome this issue, we generalize the deviation
bound to sums of two {\it orthogonal} vectors
where one of the vectors satisfies a strong $\|\cdot\|_\infty$--norm
estimate:
\begin{prop}\label{aljhfbofjhbojb}
Fix any $K\geq 1$.
For each $n$ let $c=c(n)$
be a non-random unit vector in $\R^n$, $w=w(n)$ be any fixed vector in $c^\perp$,
and let $\delta_n\in(0,1)$ be a sequence of numbers
such that $\delta_n\, n$ converges to infinity.
Assume that $\lim\limits_{n\to\infty}\big((\delta_n\, n)^{3/2}\,\|c\|_\infty\big)=0$.
For each
$n$, let $\xi_1^{(n)},\xi_2^{(n)},\dots,\xi_n^{(n)}$
be mutually independent centered $K$--subgaussian variables of
unit variances.
Then for every $\varepsilon>0$ and all large $n$ we have
$$
\Prob\bigg\{\sum_{i=1}^n (c_i+w_i)\xi_i^{(n)}
\geq (1-\varepsilon)\sqrt{\delta_n\, n}\bigg\}
\geq \exp\big(-\delta_n\, n/2\big).
$$
\end{prop}
\begin{proof}
Fix a small $\varepsilon>0$.
In view of the assumptions on $c(n)$, there is a sequence of positive
numbers $\nu(n)$ converging to zero such that
$(\delta_n\, n)^{3/2}\|c\|_\infty\leq \nu$ for all $n$.
We shall consider two scenarios:
\begin{itemize}
\item $\|c+w\|_2\geq \nu^{-1/8}\sqrt{\delta_n\, n}=\omega(\sqrt{\delta_n\, n})$.
Observe that $\Var(\sum_{i=1}^n (c_i+w_i)\xi_i^{(n)})=\|c+w\|_2^2=\omega(\delta_n\, n)$,
whereas the variable $\sum_{i=1}^n (c_i+w_i)\xi_i^{(n)}$
is centered and $K'\|c_i+w_i\|_2$--subgaussian, for some $K'$ depending only on $K$.
Lemma~\ref{alkjfnbalkfjnselfkjsnflkj} then implies that,
assuming $n$ is sufficiently large,
$$
\Prob\bigg\{\sum_{i=1}^n (c_i+w_i)\xi_i^{(n)}\geq \sqrt{\delta_n\, n}\bigg\}\geq \tilde K
=\omega\big(\exp(-\delta_n\, n/2)\big),
$$
for some $\tilde K>0$ depending only on $K$.
\item $\|c+w\|_2< \nu^{-1/8}\sqrt{\delta_n\, n}$.
Let $I(n)$ be a
$\lfloor\nu^{-1/2}\delta_n^2 n^2\rfloor$--subset of $[n]$
corresponding to $\lfloor\nu^{-1/2}\delta_n^2 n^2\rfloor$ largest
(by the absolute value) components of $c+w$, with ties resolved arbitrarily.
Further, let $y'=y'(n)\in \R^{[n]\setminus I}$ be the restriction
of $c+w$ to $[n]\setminus I$.
In view of the condition $\|c+w\|_2< \nu^{-1/8}\sqrt{\delta_n\, n}$, we have
$$
\|y'\|_\infty\leq \frac{\nu^{-1/8}\sqrt{\delta_n\, n}}{\sqrt{|I|}}
\leq (1+o(1))\nu^{1/8}(\delta_n\, n)^{-1/2}=o\big((\delta_n\, n)^{-1/2}\big).
$$
Further, in view of orthogonality of $w$ and $c$,
\begin{align*}
\sum_{i\in I}w_i c_i&=-\sum_{i\in [n]\setminus I}(y_i'-c_i) c_i\\
&\geq \sum_{i\in [n]\setminus I}c_i^2
-\|y_i'\|_2\\
&\geq 1-|I|\,\|c\|_\infty^2
-\|y_i'\|_2\\
&\geq 1-o(1)-\|y_i'\|_2,
\end{align*}
whereas,
$$
\sum_{i\in I}w_i c_i
\leq \|w\|_2\,\|c\|_\infty\,\sqrt{|I|}
\leq (1+o(1))\nu^{-1/8}\cdot \nu^{3/4}=o(1).
$$
We infer that $\|y_i'\|_2\geq 1-o(1)$. Thus, setting $y=y(n):=y'/\|y'\|_2$,
we obtain a unit vector with $\sqrt{\delta n}\|y\|_\infty=o(1)$.
Applying Proposition~\ref{aksjfnoifurhfoioin} with $\delta_n\,(1-\varepsilon/2)$
in place of $\delta_n$ and $\varepsilon/2$ in place of $\varepsilon$, we get
\begin{align*}
\Prob\bigg\{\sum_{i\in [n]\setminus I}
y_i'\xi_i^{(n)}\geq (1-\varepsilon)\sqrt{\delta_n\, n}\bigg\}
&\geq 
\Prob\bigg\{\sum_{i\in [n]\setminus I}
y_i'\xi_i^{(n)}\geq (1-\varepsilon/2)\sqrt{\delta_n\,(1-\varepsilon/2) n}\bigg\}\\
&\geq\exp\big(-\delta_n\,
(1-\varepsilon/2)n/2\big),
\end{align*}
assuming $n$ is sufficiently large.
On the other hand, again applying Lemma~\ref{alkjfnbalkfjnselfkjsnflkj},
$$
\Prob\bigg\{\sum_{i\in I} (c_i+w_i)\xi_i^{(n)}\geq 0\bigg\}\geq \tilde K,
$$ 
for some $\tilde K>0$ depending only on $K$.
Combining the last two relations, we get the desired estimate.
\end{itemize}
\end{proof}

To apply the last proposition to our setting of interest, 
we shall bound the number of constraints violated by a test vector:
\begin{corollary}\label{aksjfopifjnakmlfk}
Let $K$, $m_n$, $A=A(n)$, and $c=c(n)$ be as in Theorem~\ref{th main upper},
and let $w=w(n)$ be any sequence of vectors orthogonal to $c=c(n)$.
Then for every small constant $\varepsilon>0$ and all large $n$ we have
$$
\Big(A\Big(w+\frac{1+\varepsilon}{\sqrt{2\log(m_n/n)}}\,c\Big)\Big)_i
\geq 1+\varepsilon/2\quad
\mbox{for at least $m_n\big(\frac{n}{m_n}\big)^{1-\varepsilon/8}$ indices $i\leq m_n$}
$$
with probability at least
$1-\exp\big(-m_n\big(\frac{n}{m_n}\big)^{1-\varepsilon/5}\big)$.
\end{corollary}
\begin{proof}
In view of Proposition~\ref{aljhfbofjhbojb}, applied with
$\delta: =\frac{2\log(m_n/n)\,(1+\varepsilon/2)^2}
{(1+\varepsilon)^2(1-\varepsilon/4)^2 n}$,
for all large $n$ and every $i\leq m_n$ we have
$$
\Prob\Big\{
\Big(A\Big(w+\frac{1+\varepsilon}{\sqrt{2\log(m_n/n)}}\,c\Big)\Big)_i
\geq 1+\varepsilon/2
\Big\}
\geq \Big(\frac{n}{m_n}\Big)^{1-\varepsilon/4}.
$$
Hence, the probability that the last condition holds for less than
$m_n\big(\frac{n}{m_n}\big)^{1-\varepsilon/8}$ indices $i\leq m_n$,
is bounded above by
$$
\bigg(e\Big(\frac{n}{m_n}\Big)^{\varepsilon/8-1}\bigg)^{m_n(\frac{n}{m_n})^{1-\varepsilon/8}}
\bigg(1-\Big(\frac{n}{m_n}\Big)^{1-\varepsilon/4}\bigg)^{
m_n-m_n\big(\frac{n}{m_n}\big)^{1-\varepsilon/8}}
\leq \exp\bigg(-m_n\Big(\frac{n}{m_n}\Big)^{1-\varepsilon/5}\bigg)
$$
for all sufficiently large $n$.
\end{proof}

\medskip

Recall that the outer radius $R(P)$ of a polyhedron $P$
is defined as $R(P):=\max\limits_{x\in P}\|x\|_2$.

\begin{prop}[An upper bound on the outer radius]\label{askfjnafkjnofjnfkjnasd}
Let $K,n,m_n,A(n)$ be as in Theorem~\ref{th main upper}. Then
with probability $1-n^{-\omega(1)}$, the polyhedron
$\big\{x\in\R^n:\;Ax\leq {\bf 1}\big\}$ has the outer radius
$R(P)$ bounded above by a function of $K$.
\end{prop}
\begin{proof}
Let $\kappa=\kappa(K)>0$ be a small parameter to be defined later,
and let $\Net$ be a Euclidean $\kappa$--net on $S^{n-1}$ of size at most
$\big(\frac{2}{\kappa}\big)^n$ \cite[Chapter~4]{VershyninsBook}. Lemma~\ref{alkjfnbalkfjnselfkjsnflkj}
implies that for some $\tilde K>0$ depending
only on $K$ and for every $y'\in \Net$ and $i\leq m_n$,
$$
\Prob\big\{\langle \row_i(A),y'\rangle\geq \tilde K
\big\}\geq \tilde K.
$$
Define $u>0$ as the largest number satisfying
$$
\bigg(\frac{e}{u}\bigg)^{u}(1-\tilde K)^{1-u}\leq \exp(-\tilde K/2).
$$
Then, by the above, for every $y'\in \Net$,
\begin{align*}
\Prob\big\{\langle \row_i(A),y'\rangle<\tilde K\mbox{ for at least $1-u$ fraction of 
indices $i\in[m_n]$}\big\}
&\leq \bigg(\frac{e}{u}\bigg)^{um_n}(1-\tilde K)^{m_n-um_n}\\
&\leq \exp(-\tilde K m_n/2).
\end{align*}
Setting $\kappa:=\frac{\tilde K\sqrt{u}}{8}$ and using the above bound on the
size of $\Net$, we obtain that
the event
$$
\{\|A\|\leq 2\sqrt{m_n}\}
\cap
\bigcap\limits_{y'\in\Net}\big\{\langle \row_i(A),y'\rangle\geq 
\tilde K\mbox{ for at least $u m_n$ 
indices $i\in[m_n]$}\big\}
$$
has probability $1-n^{-\omega(1)}$
(recall that the $\|A\|\leq 2\sqrt{m_n}$ with probability $1-\exp(-\Omega(m_n))$;
see \cite[Chapter~4]{VershyninsBook}).
Condition on any realization of $A$ from the above event.
Assume that there exists $y\in (2\tilde K^{-1} S^{n-1})\cap P$.
Let $y'\in\Net$ be a vector satisfying $\big\|\frac{\tilde K}{2} y-y'\big\|_2\leq\kappa$.
In view of the conditioning, there are at least $u m_n$ 
indices $i\in[m_n]$
such that $\langle \row_i(A),y'-y\rangle\geq \tilde K/2$. Consequently,
$$
\|A(y-y')\|_2\geq \frac{\tilde K}{2}\sqrt{um_n}.
$$
Since $\|A\|\leq 2\sqrt{m_n}$, we get
$$
\|y-y'\|_2\geq \frac{\tilde K\sqrt{u}}{4},
$$
contradicting our choice of $\kappa$.
The contradiction shows that the outer radius of $P$ is at most $2/\tilde K$,
and the claim is verified.   
\end{proof}

\subsection{Proof of Theorem~\ref{th main upper}, and an upper bound on $\cw(P)$}

\begin{proof}[Proof of Theorem~\ref{th main upper}]
In view of Proposition~\ref{askfjnafkjnofjnfkjnasd},
with probability $1-n^{-\omega(1)}$
the outer radius of $P$ is bounded above by $\tilde K\geq 1$,
for some $\tilde K$ depending only on $K$.

\medskip

Fix any small $\varepsilon>0$.
By the above,
to complete the proof of the theorem it is sufficient
to show that with probability $1-n^{-\omega(1)}$,
every vector $z$ orthogonal to $c$ and having the Euclidean norm at most $2\tilde K$,
satisfies
$$
z+\frac{1+\varepsilon}{\sqrt{2\log(m_n/n)}}\,c\notin P.
$$
Let $k$ be the largest integer bounded above by
$$
m_n\Big(\frac{n}{m_n}\Big)^{1-\varepsilon/8}.
$$
Define $\delta_n:=\exp(-(m_n/n)^{K'})$,
where $K'>0$ is a sufficiently small constant depending only on $\varepsilon$ and $K$
(the value of $K'$ can be extracted from the argument below).
Let $\Net'$ be a Euclidean $\delta_n$--net
in $(2\tilde K\, B_2^n)\cap c^\perp$.
It is known that $\Net'$ can be chosen to have cardinality at most
$\big(\frac{4\tilde K}{\delta_n}\big)^n$ (see \cite[Chapter~4]{VershyninsBook}).

\medskip

Fix for a moment any $z'\in \Net'$, and observe that, in view of
Corollary~\ref{aksjfopifjnakmlfk} and the definition of $k$, we have
\begin{equation}\label{alskfjalkfjnalkfjn}
\Big(A\Big(z'+\frac{1+\varepsilon}{\sqrt{2\log(m_n/n)}}\,c\Big)\Big)_i
\geq 1+\varepsilon/2\quad
\mbox{for at least $k$ indices $i\leq m_n$}
\end{equation}
with probability at least
$$
1-\exp\Big(-m_n\Big(\frac{n}{m_n}\Big)^{1-\varepsilon/5}\Big).
$$
Therefore, assuming that $K'$ is sufficiently small and taking the union
bound over all elements of $\Net'$, we obtain
that with probability $1-n^{-\omega(1)}$, 
\eqref{alskfjalkfjnalkfjn} holds simultaneously for all $z'\in\Net'$.

\bigskip

Condition on any realization of $A$ such that the last event holds
and, moreover, such that $\|A\|\leq 2\sqrt{m_n}$.
Assume for a moment that there is $z\in (2\tilde K\, B_2^n)\cap c^\perp$ such that
$$
\Big(A\Big(z+\frac{1+\varepsilon}{\sqrt{2\log(m_n/n)}}\,c\Big)\Big)_i
\leq 1\quad\mbox{for all $i\leq m_n$}.
$$
In view of the above, it implies that there is a vector $z'\in\Net'$
at distance at most $\delta_n$ from $z$ such that
$$
(A(z'-z))_i\geq \frac{1}{2}\varepsilon
$$
for at least $k$ indices $i\leq m_n$.
Hence, the spectral norm of $A$ can be estimated as
$$
\|A\|\geq
\frac{\varepsilon\sqrt{k}}{2\delta_n}>2\sqrt{m_n},
$$
contradicting the bound $\|A\|\leq 2\sqrt{m_n}$.
The contradiction implies the result.
\end{proof}

The next corollary constitutes an upper bound in Corollary~\ref{mw cor nongauss}.
Since its proof is very similar to that of Corollary~\ref{mw cor lower},
we only sketch the argument.
\begin{corollary}
Let $K, n, m, A(n)$ be as in Theorem~\ref{th main upper},
and assume additionally that $\log^{3/2}(m_n/n)=o(\sqrt{n/\log n})$.
Then, $\limsup\limits_{n\to\infty}
\sqrt{2\log(m_n/n)} \cw(P)\leq 2$ almost surely.
\end{corollary}
\begin{proof}
In view of Proposition~\ref{askfjnafkjnofjnfkjnasd}, there is $\tilde K>0$
depending only on $K$ such that the event
$$
\tilde \Event:=\big\{\limsup\limits_{n\to\infty} R(P(n))\leq \tilde K\big\}
$$
has probability one, where, as before, $R(P(n))$ denotes the
outer radius of the polyhedron $P(n)$.
Set
$$
\Event_\delta
:=\big\{\limsup\limits_{n\to\infty} \sqrt{2\log(m_n/n)}\cw(P(n))\geq 2+\delta\big\}\cap
\tilde \Event,
\quad \delta>0.
$$
We will prove the corollary by contradiction.
Assume that the assertion of the corollary does not hold.
Then, in view of the above definitions, there is $\delta>0$
such that
$$
\Prob\big(\Event_\delta\big)\geq \delta.
$$
Condition for a moment on any
realization of the polyhedra $P(n)$, $n\in \N$, from $\Event_\delta$,
and let $(n_k)_{k=1}^\infty$ be an increasing sequence of integers
such that $\sqrt{2\log(m_{n_k}/n_k)}\cw(P(n_k))\geq 2+\delta/2$ and
$R(P(n_k))\leq 2\tilde K$ for every $k$.
Assume that $c(n)$, $n\geq 1$, are mutually independent uniform random unit vectors
which are also independent from the matrices $\{A(n),\;n\in\N\}$.
The conditions
$$
\cw(P(n_k))\geq
\frac{2+\delta/2}{\sqrt{2\log(m_{n_k}/n_k)}}\quad \mbox{and}\quad
R(P(n_k))\leq 2\tilde K
$$
imply that for some $\tilde \delta>0$ depending only on $\delta$
and $\tilde K$,
$$
\Prob_c\bigg\{
\max\{\langle c(n_k),x\rangle:\;x\in P(n_k)\}> \frac{1+\tilde\delta}{\sqrt{2\log(m_{n_k}/n_k)}}
\bigg\}
\geq \tilde \delta,
$$
where $\Prob_c$ is the conditional probability given the realization of $P(n)$, $n\in\N$,
from $\Event_\delta$.
By the Borel--Cantelli lemma, the last assertion yields
$$
\Prob_c\bigg\{
\limsup\limits_{n\to\infty}
\sqrt{2\log(m_n/n)}\max\{\langle c(n),x\rangle:\;x\in P(n)\}\geq 1+\tilde\delta
\bigg\}=1.
$$
Removing the conditioning on $\Event_\delta$, we arrive at the estimate
\begin{equation}\label{alsfiuheofiuyhofifas}
\Prob\bigg\{
\limsup\limits_{n\to\infty}
\sqrt{2\log(m_n/n)}\max\{\langle c(n),x\rangle:\;x\in P(n)\}\geq 1+\tilde \delta
\bigg\}\geq \tilde \delta.
\end{equation}
The standard concentration inequality on the sphere
\cite[Chapter~5]{VershyninsBook} and the conditions on $m_n$ imply that
$\log^{3/2}(m_n/n)\|c\|_\infty=o(1)$ with probability $1-n^{-\omega(1)}$,
and therefore, by Theorem~\ref{th main upper},
$$
\Prob\big\{\sqrt{2\log(m_n/n)}\max\{\langle c(n),x\rangle:\;x\in P(n)\}
\leq 1+\tilde \delta/2
\big\}\geq 1-n^{-\omega(1)}.
$$
The latter contradicts \eqref{alsfiuheofiuyhofifas}, and the result follows.
\end{proof}

\section{The Gaussian setting}\label{section Gaussian}

In this section, we prove Theorem~\ref{th main gauss} and Corollary~\ref{mw cor gauss}.
Note that in view of rotational invariance of the Gaussian distribution,
we can assume without loss of generality that the cost vectors
in Theorem~\ref{th main gauss} are of the form $c(n)=(\frac{1}{\sqrt{n}},
\dots,\frac{1}{\sqrt{n}})$.
Thus, in the regime $\log^{3/2}m_n=o(\sqrt{n})$, the statement is already
covered by the
more general results of Sections~\ref{alksjfalkfjlskjnl}
and~\ref{akjshbfakjhfbkqjwhb}.

\subsection{The outer radius of random Gaussian polytopes}\label{outersection}

In this subsection, we consider the bound on the outer radius of random Gaussian
polytopes.
The next result immediately implies upper bounds on $z^*$ and $\cw(P)$
in Theorem~\ref{th main gauss} and Corollary~\ref{mw cor gauss}
in the entire range $m_n=\omega(n)$. 
We anticipate that
the statement is known although we are not aware of a reference,
and for that reason provide a proof.
\begin{prop}[Outer radius of the feasible region in the Gaussian setting]\label{prop main 3}
Let $m_n$ satisfy $\lim\limits_{n\to\infty}\frac{m_n}{n}=\infty$,
and for every $n$ let $A$ be an $m_n\times n$ random matrix
with i.i.d standard Gaussian entries.
Then, denoting by $R(P(n))$ the outer radius of the feasible region,
for any constant $\varepsilon>0$ we have
$$
\Prob\big\{\sqrt{2\log(m_n/n)}\,R(P(n)) \leq 1+\varepsilon\big\}
\geq 1-n^{-\omega(1)}.
$$
In particular,
$\limsup\limits_{n\to\infty} \sqrt{2\log(m_n/n)}\,R(P(n)) \leq 1$ almost surely.
\end{prop}

\bigskip

The following approximation of the standard Gaussian distribution
is well known (see, for example, \cite[Volume I, Chapter VII, Lemma 2]{feller1991introduction}):
\begin{lemma}
Let $g$ be a standard Gaussian variable.
Then for every $t>0$,
$$
\Big(\frac{1}{t}-\frac{1}{t^3}\Big)\exp(-t^2/2)\leq 
\Prob\{g\geq t\}\leq \frac{1}{t}\exp(-t^2/2).
$$
\end{lemma}
Let $G$ be a standard Gaussian random vector in $\R^{m_n}$, and $G_{(1)} \leq G_{(2)}\leq \dots\leq G_{(m_n)}$
be the non-decreasing rearrangement of its components.
Let $1\leq k\leq m_n/2$. Then, from the above lemma, for every $t>0$ we have 
\begin{equation}\label{aowasdfiufnofinogugbowu}
\Prob\big\{G_{(m_n-k)} \leq t\big\}\leq {m_n\choose k}\Prob\{g\leq t\}^{m_n-k}
\leq \Big(\frac{em_n}{k}\Big)^k
\Big(1-\Big(\frac{1}{t}-\frac{1}{t^3}\Big)\exp(-t^2/2)\Big)^{m_n-k}.
\end{equation}
As a corollary, we obtain the following deviation estimate:
\begin{corollary}\label{aljhsfaberoguehbouerh}
Fix any $\varepsilon\in(0,1)$.
Let sequences of positive integers $m_n$, $n\geq 1$, and
$k_n$, $n\geq 1$,
satisfy $m_n=\omega(n)$, $k_n=\omega(1)$, and $k_n=o(m_n)$.
Further, for each $n$, let $G=G(n)$ be a standard Gaussian vector in $\R^{m_n}$.
Then
$$
\Prob\big\{G_{(m_n-k_n)} \leq (1-\varepsilon)\sqrt{2\log (m_n/k_n)}\big\}=
\exp\Big(-\Big(\frac{m_n}{k_n}\Big)^{\Omega(1)}k_n\Big).
$$
\end{corollary}
\begin{proof} 
Applying \eqref{aowasdfiufnofinogugbowu}, we get
\begin{align*}
\Prob&\big\{G_{(m_n-k_n)} \leq (1-\varepsilon)\sqrt{2\log (m_n/k_n)}\big\}\\
&\leq 
\Big(\frac{em_n}{k_n}\Big)^{k_n}
\bigg(1-\Big(\frac{1}{(1-\varepsilon)(2\log (m_n/k_n))^{1/2}}
-\frac{1}{(1-\varepsilon)^3(2\log (m_n/k_n))^{3/2}}
\Big)
\Big(\frac{k_n}{m_n}\Big)^{(1-\varepsilon)^2}\bigg)^{m_n-k_n}\\
&=\exp\Big(-\Big(\frac{m_n}{k_n}\Big)^{\Omega(1)}k_n\Big),
\end{align*}
implying the claim.
\end{proof}

\begin{proof}[Proof of Proposition~\ref{prop main 3}]
Fix any $\varepsilon>0$.
Observe that the statement is equivalent to showing that with probability
$1-n^{-\omega(1)}$
the matrix $A$ has the property that for
every $y\in S^{n-1}$ there is $i\leq m_n$ with
$(Ay)_i\geq (1-\varepsilon)\sqrt{2\log (m_n/n)}$.

The argument below is similar to the proof of Proposition~\ref{askfjnafkjnofjnfkjnasd}.
Let $k_n$ be the largest integer bounded above by
$$
m_n\Big(\frac{n}{m_n}\Big)^{1-\varepsilon/8}.
$$
Define $\delta_n:=\exp(-(m_n/n)^{K})$,
where $K>0$ is a sufficiently small constant depending only on $\varepsilon$.
Let $\Net$ be a Euclidean $\delta_n$--net
in $S^{n-1}$ of cardinality at most
$\big(\frac{2}{\delta_n}\big)^n$ (see \cite[Chapter~4]{VershyninsBook}).

\bigskip

Fix for a moment any $y'\in \Net$, and observe that, in view of
Corollary~\ref{aljhsfaberoguehbouerh} and the definition of $k_n$, we have
\begin{equation}\label{akjnfoiejnfoijn}
(Ay')_i\geq (1-\varepsilon/2)\sqrt{2\log (m_n/n)}\quad
\mbox{for at least $k_n$ indices $i\leq m_n$}
\end{equation}
with probability at least
$1-\exp\big(-\big(\frac{m_n}{k_n}\big)^{\Omega(1)}k_n\big)\geq 1-\exp(-k_n)$.
Therefore, assuming that $K$ is sufficiently small and taking the union
bound over all elements of $\Net$, we obtain
that with probability $1-n^{-\omega(1)}$, 
\eqref{akjnfoiejnfoijn} holds simultaneously for all $y'\in\Net$.

\bigskip

Condition on any realization of $A$ such that the last event holds
and, moreover, such that $\|A\|\leq 2\sqrt{m_n}$ (recall
that the latter occurs with probability $1-\exp(-\Omega(m_n))$;
see \cite[Chapter~4]{VershyninsBook}).
Assume for a moment that there is $y\in S^{n-1}$ such that
$(Ay)_i\leq (1-\varepsilon)\sqrt{2\log (m_n/n)}$ for all $i\leq m_n$.
In view of the above, it implies that there is a vector $y'\in\Net$
at distance at most $\delta_n$ from $y$ such that
$$
(A(y'-y))_i\geq \frac{1}{2}\varepsilon\sqrt{2\log (m_n/n)}
$$
for at least $k_n$ indices $i\leq m_n$.
Hence, the spectral norm of $A$ can be estimated as
$$
\|A\|\geq
\frac{\varepsilon\sqrt{2k_n\log (m_n/n)}}{2\delta_n}>2\sqrt{m_n},
$$
contradicting the bound $\|A\|\leq 2\sqrt{m_n}$.
The contradiction shows that for every $y\in S^{n-1}$
there is $i\leq m_n$ such that $(Ay)_i\geq (1-\varepsilon)\sqrt{2\log (m_n/n)}$.
The result follows.
\end{proof}

\subsection{Proof of Theorem~\ref{th main gauss} and Corollary~\ref{mw cor gauss}}

As we have noted, the upper bounds in Theorem~\ref{th main gauss} and Corollary~\ref{mw cor gauss}
follow readily from Proposition~\ref{prop main 3},
and so we only need to verify the lower bounds.
In view of results of Section~\ref{alksjfalkfjlskjnl}
which also apply in the Gaussian setting, we may assume that
$m_n$ satisfies $m_n=n^{\omega(1)}$.

\medskip

Regarding the proof of Theorem~\ref{th main gauss},
since $Ac$ is a standard Gaussian vector, it follows 
from the formula for the Gaussian density that
$$
\Prob\big\{\|Ac\|_\infty\leq (1+\varepsilon)\sqrt{2\log m_n}\big\}
\geq 1-m_n^{-\Omega(1)}=1-n^{-\omega(1)},
$$
whereas, by the assumption $m_n=n^{\omega(1)}$, we have $\sqrt{2\log m_n}=(1+ o(1))\sqrt{2\log (m_n/n)}$.
This implies the result.

\medskip

The proof of the lower bound in Corollary~\ref{mw cor gauss}
follows exactly the same scheme as the first part of the
proof of Corollary~\ref{mw cor lower}.

\section{Numerical experiments}\label{numericssection}

Our results described in the previous sections
naturally give rise to the questions:
(a) {\it How close to each other are the asymptotic estimates of the optimal objective
value and
empirical observations?}, (b)
{\it What is the asymptotic distribution and standard deviation of $z^*$?}, and
(c) {\it How does the algorithm in the proof of Theorem ~\ref{th main lower}  perform in practice?} We discuss these questions in the following subsections, and make conjectures while providing numerical evidence.
We use Gurobi 10.0.1 for solving instances of the LP (\ref{lp_max}), and set all Gurobi parameters to default. 

\subsection{Magnitude of the optimal objective value}



Below, we investigate the quality of approximation of
the optimal objective value $z^*$ by the asymptotic bound
$ab := (2\log(m/n))^{-\frac{1}{2} }$ given in Theorems~\ref{th main nongauss}
and~\ref{th main gauss}. 
The empirical outcomes are obtained through multiple runs of the LP \eqref{lp_max} under various choices of parameters.
We consider two distributions of the entries of $A$:
either $A$ is standard Gaussian or its entries are i.i.d Rademacher ($\pm 1$)
random variables.
In either case and for different choices of $m$, we sample the random LP \eqref{lp_max} taking the sample size $50$ and letting $c$ be
the vector of i.i.d Rademacher variables rescaled to have unit Euclidean norm.
The cost vector is generated once and remains the same for 
each of the $50$ instances of the LP within a sample.

\begin{table}[h!]
    \centering
    \begin{tabular}{cccccc}
$m$  & $n$ & $ab$ & $\hat{\mu}$ & relative gap($\%$)\\
\hline  \\
1000 & 50 & 0.40853 & 0.50626 &23.92 \\
2000 & 50 & 0.36816 & 0.44119 &19.83 \\
6000 & 50 & 0.32317 & 0.37256 &15.28 \\
10000 & 50& 0.30719 & 0.35176 &14.50 \\
20000 & 50& 0.28888 & 0.32473 &12.41 \\

    \end{tabular}
    \caption{Asymptotic versus empirically observed 
    optimal objective value in the Gaussian setting (the sample size = $50$)}
    \label{tab: opt_obj_nonAsym}
\end{table}
As is shown in Tables~\ref{tab: opt_obj_nonAsym}
and~\ref{tab: opt_obj_nonAsym_Rade}, for a fixed value of $n=50$, as the number of constraints $m$ progressively increases in magnitude,
the {\it relative gap}
between $ab$ and the sample mean
$\hat{\mu}$ of the optimal objective values,
defined as
$$\left(\frac{ | ab -  \hat{\mu} | }{ ab } \times 100 \right)\%, $$
decreases. 

\begin{table}[h!]
    \centering
    \begin{tabular}{cccccc}
$m$  & $n$ & $ab$ & $\hat{\mu}$ & relative gap($\%$)\\
\hline  \\
1000 & 50 & 0.40853 & 0.50369  & 23.29 \\
2000 & 50 & 0.36816 & 0.43801 & 18.97 \\
6000 & 50 & 0.32317 & 0.37200 & 15.11 \\
10000 & 50 & 0.30719 & 0.35220 & 14.65\\
20000 & 50 & 0.28888 & 0.32856 & 13.73\\

    \end{tabular}
    \caption{Asymptotic versus empirically observed 
    optimal objective value in the Rademacher setting (the sample size = $50$)}
    \label{tab: opt_obj_nonAsym_Rade}
\end{table}

\subsection{Structural assumptions on the cost vector}\label{secstructural}
In this subsection, we carry out a numerical study of the
technical conditions on the cost vectors from Theorem~\ref{th main nongauss}.
Roughly speaking, those conditions stipulate that
the cost vector has significantly more than $\log^{3/2}(m/n)$
``essentially non-zero'' components. Since Theorem~\ref{th main nongauss}
is an asymptotic result, it is not at all clear what practical
assumptions on the cost vectors should be imposed to guarantee
that the resulting optimal objective value is close to the asymptotic bound.

\begin{table}[h!]
    \centering
    \begin{tabular}{ccc}
        $k$ & $\hat{\mu}(k)$ & relative gap ($\%$) \\
        \hline \\
1 & 0.429907 & 15.08\\
2 & 0.454357 & 10.25\\
3 & 0.459387 & 9.25\\
4 & 0.479007 & 5.38\\
5 & 0.481839 & 4.82\\
6 & 0.486141 & 3.97\\
7 & 0.479635 & 5.25\\
8 & 0.493168 & 2.58\\
9 & 0.490829 & 3.04\\
10 & 0.498035& 1.62\\
    \end{tabular}
    \caption{The sample mean
    of the optimal objective value corresponding to $c(n,k)$ for a given value of 
    the parameter $k$. The random matrix $A$
    has dimensions $m=1000$ and $n=50$, with each entry
    equidistributed with
    the product of Bernoulli($1/2$) and N($0$,$2$) variables. The sample size = $50$.
    }
    \label{tab:cost_obj}
\end{table}

In the following experiment, we consider a random coefficient matrix $A$
with i.i.d subgaussian components equidistributed with
the product of Bernoulli($1/2$) and N($0$,$2$) random variables.
Note that with this definition the entries have zero mean and unit variance.
Further, we consider a family of cost vectors $c(n,k)$ parametrized
by an integer $k$, so that the vector $c(n,k)$
has $k$ non zero components equal to $1/\sqrt{k}$ each, and the rest
of the components are zeros.
For a given choice of $k$, we sample the LP \eqref{lp_max}
with the above distribution of the entries and the cost vector
$c(n,k)$.
Our goal is to compare the resulting sample mean of the optimal objective value
to the sample mean obtained in the previous subsection for the same matrix dimensions
and for the strongly incompressible cost vector.
We fix $m=1000$ and $n=50$, and let $\hat{\mu}(k)$ to be the sample mean of optimal objective value of LP(\ref{lp_max}) with the cost vector $c(n,k)$. We denote the sample mean of objective value with the strongly incompressible cost vector by $\tilde{\mu}$.
Our empirical results show that, as long as $k$ is small (so that
$c(n,k)$ is very sparse), the value of the sample mean
differs significantly from the one in Table~\ref{tab: opt_obj_nonAsym}.
On the other hand, for $k$ sufficiently large the relative gap between the sample means defined by 
$$\left(\frac{ | \tilde{\mu} -  \hat{\mu}(k) | }{ \tilde{\mu} } \times 100 \right)\% $$
becomes negligible.
The results are presented in Table~\ref{tab:cost_obj}.

\subsection{Limiting distribution of the optimal objective value}

Recall that, given
mutually independent standard Gaussian variables,
their maximum converges (up to appropriate rescaling and translation)
to the Gumbel distribution as the number of variables tends to infinity.
If the vertices of the random polyhedron $P=\{x\in\R^n:\;Ax\leq {\bf 1}\}$
were behaving like independent Gaussian vectors (with a proper rescaling)
then the limiting distribution of $z^*$ for any given fixed
cost vector would be asymptotically Gumbel.
Our computations suggest that in fact the limiting distribution is Gaussian.

\begin{figure}[h!] 
  \centering
  \subfloat{\includegraphics[width=0.4\textwidth]{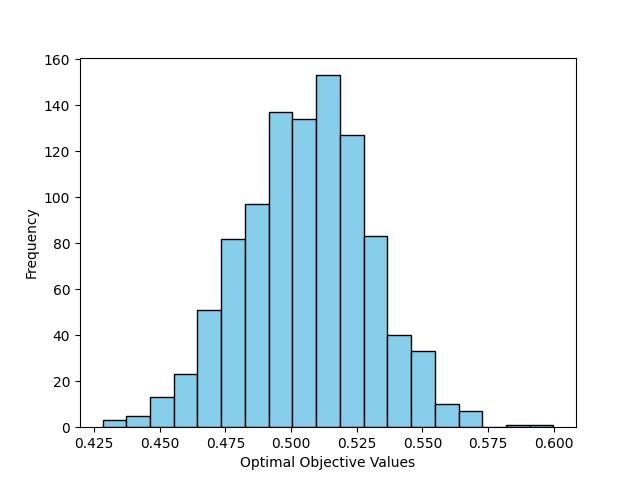}}\label{fig:f1g}
  \subfloat{\includegraphics[width=0.4\textwidth]{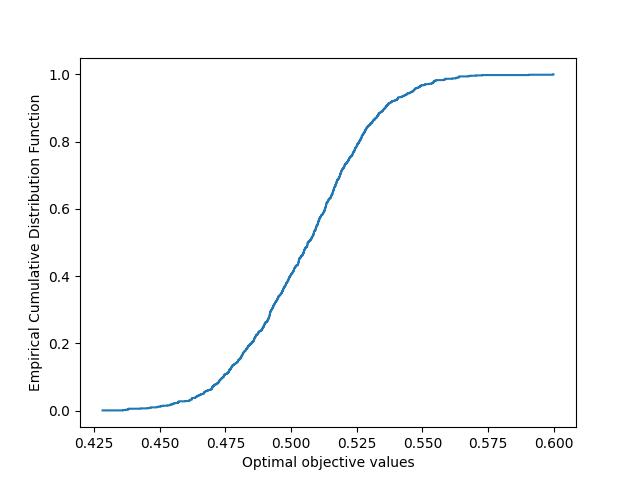}}\label{fig:f2g}
  \caption{The frequency histogram and the empirical cumulative distribution function of the optimal objective values for the Gaussian matrix $A$ with $m=1000$, $n=50$
  (the sample size = $1000$). The KS test results: KS statistic = $0.0232$, $p$-value =
  $0.6453$.}
  \label{fig: gaussian}
\end{figure}

\begin{figure}[h!] 
  \centering
  \subfloat{\includegraphics[width=0.4\textwidth]{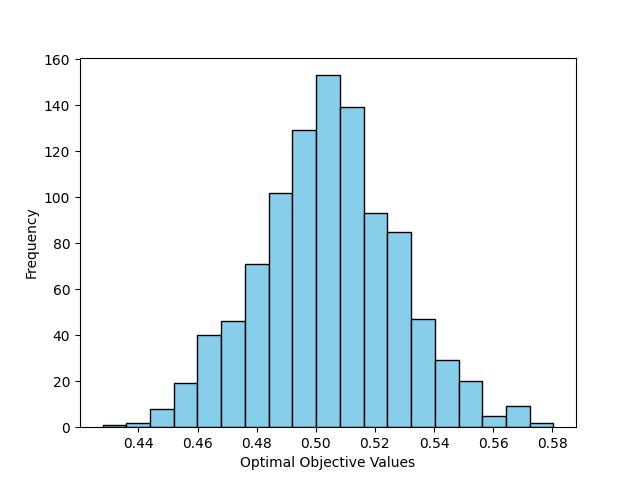}}\label{fig:f1r}
  \subfloat{\includegraphics[width=0.4\textwidth]{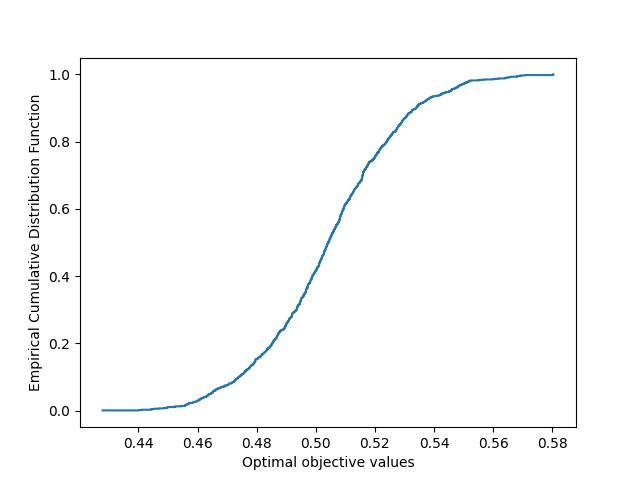}}\label{fig:f2r}
  \caption{The frequency histogram and the empirical cumulative distribution function of the optimal objective values for the Rademacher matrix $A$ with $m=1000$, $n=50$
  (the sample size = $1000$).
  The KS test results: KS statistic = $0.0219$, $p$-value = $0.7161$.}
  \label{fig: Rademacher}
\end{figure}

\bigskip

In our numerical experiments, we consider Gaussian and Rademacher
random coefficient matrices of dimensions $m=1000$ and $n=50$.
The sample size in either case is taken to be $1000$.
As in the previous experiment, we take $c$
to be a random vector with i.i.d Rademacher components rescaled
to have the Euclidean norm one. We generate the cost vector once so that
it is the same for every LP from a sample.
We employ the Kolmogorov--Smirnov (KS)
test to compare the sample distribution of $z^*$ to Gaussian
of a same mean and variance. 
In either case, the obtained p-value exceeds the conventional significance level of 0.05, indicating lack of substantial evidence to refute the null hypothesis that the data is normally distributed.
Further, visual representation of the frequency histogram and the empirical cumulative distribution function of the optimal values shown in Figures~\ref{fig: gaussian}
and~\ref{fig: Rademacher} validate this conjecture.

\bigskip

The next conjecture summarizes the experimental findings:
\begin{conjecture}[Limiting distribution of $z^*$]
Let $m$, $n$, $A(n)$ be as in Theorem~\ref{th main nongauss},
and for each $n$
let $c=c(n)$ be uniform random unit cost vector independent from $A(n)$.
Then the sequence of normalized random variables
$$
\frac{z^*-\Exp\,z^*}{\sqrt{\Var\,z^*}}
$$
converges in distribution to the standard Gaussian.
\end{conjecture}

\subsection{Standard deviation of the optimal objective value}

As in the previous two numerical experiments, in this one we consider two
types of the entries distributions: Gaussian and Rademacher.
The cost vector is generated according to the same procedure as above.
We sample the LP \eqref{lp_max} taking the sample size to be $50$,
and compute the sample standard deviation $\hat{\sigma}$ (the square
root of the sample variance) of the optimal objective value.
Based on our numerical studies we speculate
that standard deviation of $z^*$ is roughly of the order $1/\sqrt{m}$,
at least in the setting where $m$ is very large compared to $n$. 
Indeed, upon examination of the last column in Tables~\ref{tab: std_dev_gauss}
and~\ref{tab: std_dev_pmone}, we observe a consistent phenomenon:
the quantities 
are of order $1$ for various choices of $m$ and $n$. 

\begin{table}[h!]
    \centering
    \begin{tabular}{ccccc}
$m$  & $n$ & $ab$ & $\hat{\sigma}$ & $\hat{\sigma} \times \sqrt{m}$ \\
\hline  \\
1000 & 50 & 0.408539 &  0.02420 & 0.765 \\
2000 & 50 & 0.368161 &  0.01728 & 0.773 \\
4000 & 50 & 0.337791 &  0.01333 & 0.843 \\
6000 & 50 & 0.323170 &  0.01142 &  0.885 \\
8000 & 50 & 0.313877 &  0.01059 & 0.947 \\
10000 & 50 & 0.307196 & 0.01027 & 1.027 \\
1000 & 100 & 0.465991&  0.03014 & 0.953 \\
2000 & 100 & 0.408539 & 0.01592 & 0.712 \\
4000 & 100 & 0.368161 & 0.01172 & 0.742 \\
6000 & 100 & 0.349456 & 0.00990 & 0.767 \\
1000 & 150 & 0.528257 & 0.03192 & 1.010 \\
2000 & 150 & 0.441515 & 0.015066 & 0.674 \\
4000 & 150 & 0.391745 & 0.01155 & 0.731 \\
6000 & 150 & 0.368161 & 0.011850 & 0.918 \\

    \end{tabular}
    \caption{Sample standard deviation of the optimal objective value with Gaussian matrix $A$ (the sample size = $50$)}
    \label{tab: std_dev_gauss}
\end{table}

\begin{table}[h!]
    \centering
    \begin{tabular}{ccccc}
$m$  & $n$ & $ab$ & $\hat{\sigma}$ & $\hat{\sigma} \times \sqrt{m}$ \\
\hline  \\
1000 & 50 & 0.408539 & 0.020420 &  0.645 \\
2000 & 50 & 0.368161 & 0.017393 & 0.777 \\
4000 & 50 & 0.337791 & 0.011962 & 0.756 \\
6000 & 50 & 0.323170 & 0.010207 & 0.790 \\
8000 & 50 & 0.313877 & 0.010333 & 0.924 \\
10000 & 50 & 0.307196 & 0.008997 & 0.899 \\
1000 & 100 & 0.465991 & 0.023097 & 0.730 \\
2000 & 100 & 0.408539 & 0.014648 & 0.655 \\
4000 & 100 & 0.368161 & 0.010436 & 0.660 \\
6000 & 100 & 0.349456 & 0.011791 & 0.913 \\
1000 & 150 & 0.528257 & 0.032600 & 1.030 \\
2000 & 150 & 0.441515 & 0.015466 & 0.691 \\
4000 & 150 & 0.391745 & 0.011030 & 0.697 \\
6000 & 150 & 0.368161 & 0.009153 & 0.708 \\

    \end{tabular}
    \caption{The sample standard deviation of the optimal objective value with Rademacher matrix $A$ (the sample size = $50$)}
    \label{tab: std_dev_pmone}
\end{table}
\subsection{An algorithm for finding a near-optimal feasible solution}

As previously discussed, the proof of Theorem \ref{th main lower} involves an iterative projection process to restore feasibility from an initial point. 
In the process, we discretize the continuous space spanned by violated constraints at each iteration (i.e use a covering argument) as a means to apply probabilistic union bound. Recall that the goal of Algorithm~\ref{alg proof 3.1} described in the proof of Theorem \ref{th main lower} is to provide  theoretical guarantees for $z^*$. An exact software implementation of that algorithm is heavy in terms of both the amount of code and computations. 
On the other hand, it is of considerable interest to measure performance of block-iterative projection methods in the context of finding near optimal solution to instances of the random LP \eqref{lp_max}.
To approach this task, we consider a ``light'' version of the algorithm from the proof of Theorem \ref{th main lower}, which does not involve any discretizations (see Algorithm \ref{alg 1} below). 
The algorithm is a version of the classical Kaczmarz block-iterative projection method as presented in \cite{elfving1980block} and its randomized counterparts in \cite{NT2014} and \cite{ZL2023}. In our context, the blocks
correspond to the rows of the matrix $A$ that are violated by the initial point and its updates at each iteration. 
Table \ref{tab:feas sol} shows the results of running the algorithm for various parameters $m$ and $n$. For each instance of the LP, $r$ is the the number of iterations required
to restore feasibility, $x_0 = (2\log(m/n))^{-\frac{1}{2} } c$ is the initial point and $x$ is the output of the Algorithm \ref{alg 1}. Note that the obtained perturbations to $x_0$ are in the orthogonal complement of the cost vector $c$ and as a result $\|x_0\|_2 = z(x)$. The numbers $|I_0|$ and $|I_1|$ 
are
the number of violated constraints by $x_0$ at iteration 1 and $x_0 + x_1$ at iteration 2, respectively. Note that
these numbers are rather small compared to the size of the LP instance. 
We note that in our experiment, the algorithm did not converge for the $20000\times 50$
instance of the LP,
which we attribute to the random nature of the problem.
At the same time, the Algorithm converged fast due to small number of constraint violations
for all other given instances. 

\bigskip

\begin{algorithm}[H]
    \SetAlgoLined
    \KwIn{ $A$: $m\times n$ standard Gaussian random matrix,\\
     $\hspace{1.4cm}$ $c$: random cost vector uniformly distributed on the sphere}
    \KwOut{ $x$: near-optimal feasible solution to LP \eqref{lp_max}}
    $x \leftarrow (2\log(m/n))^{-\frac{1}{2} } c$\;
    $j \leftarrow 0$\;
    $\epsilon \leftarrow 0.1$\;
    \While{$x$ is not feasible}{
    $j \leftarrow j+1$\;
    Find the set of constraints $I_{j-1}$ that are violated or nearly violated
    by $x$, i.e all indices $i$ with $\langle\row_i(A),x \rangle >1 - \epsilon$\;
    Compute vector $b$ where $b_i = 1 - \epsilon - \langle\row_i(A),x \rangle $
    for all $i\in I_{j-1}$\;
    Project the constraints into the space $c^{\perp}$\  (i.e define $v_i = \row_i(A) - \langle \row_i(A) , c \rangle c$ for all $i\in I_{j-1}$) \;
    Find $x_j$ as a linear combination of $v_i$'s, $i\in I_{j-1}$, so that the constraints violated by $x$ are repaired (specifically, solve the linear system $M^TMu=b$ for $u$, then solve $x_j=Mu$ where $M$ is the matrix with columns $v_i$)\;
    $\epsilon \leftarrow \epsilon * 0.1$\;
    $x \leftarrow x+x_j$\;
    }
    return $x$

    \caption{Algorithm for finding near-optimal feasible solution to LP \eqref{lp_max}, an instance of block-iterative Kaczmarz projection method}
    \label{alg 1}
\end{algorithm}

\begin{table}[h!]
    \centering
    \begin{tabular}{cccccc}
$m$& $n$ & $r$ & $\|x_0\|_2 = z(x)$ & $|I_0|$  & $|I_1|$\\
\hline \\
1000 & 50 &   2 & 0.408539 & 11  & 1 \\ 
1000 & 100 &  2 & 0.465991 & 28  & 4 \\ 
2000 & 50 &   2 & 0.368161 & 6 & 1 \\ 
2000 & 100 &  2 & 0.408539 & 23  & 1 \\ 
4000 & 50 &   2 & 0.337791 & 13  & 2 \\ 
4000 & 100 &  2 & 0.368161 & 30  & 1 \\ 
6000 & 50 &   2 & 0.323170 & 22  & 21 \\ 
6000 & 100 &  2 & 0.349456 & 31  & 8 \\ 
8000 & 50 &   2 & 0.313877 & 22  & 6 \\ 
8000 & 100 &  2 & 0.337791 & 21  & 5 \\ 
10000 & 50 &  2 & 0.307196 & 19  & 1 \\ 
10000 & 100&  1 & 0.329505 & 34  & 0 \\ 
20000 & 100&  1 & 0.307196 & 29  & 0 \\ 
40000 & 50 &  1 & 0.273493 & 12  & 0 \\ 
40000 & 100&  2 & 0.288881 & 34  & 3 \\ 
60000 & 50 &  2 & 0.265558 & 16  & 2 \\ 
60000 & 100&  1 & 0.279576 & 36  & 0 \\ 
80000 & 50 &  1 & 0.260329 & 21  & 0 \\ 
80000 & 100&  2 & 0.273493 & 36  & 1 \\ 
100000 & 50&  2 & 0.256479 & 28  & 5 \\ 
100000 &100&  2 & 0.269040 & 35  & 1 \\

    \end{tabular}
    \caption{Numerical results for the Algorithm \ref{alg 1} for various $m$ and $n$ 
    ($r$ = the number of iterations to restore the feasibility).}
    \label{tab:feas sol}
\end{table}

\section{Conclusion}
In this paper, we study the optimal objective value $z^*$ of a random linear program where the entries of the $m\times n$ coefficient matrix $A$ have subgaussian distributions, the variables $x_1,\dots,x_n$
are free and the right hand side of each constraint is $1$. 
We establish sharp asymptotics $z^*=(1\pm o(1))(2\log(m/n))^{-\frac{1}{2} }$ in the regime $n\to\infty$, $\frac{m}{n}\to\infty$ and some additional
assumptions on the cost vectors. 
This asymptotic provides quantitative information about the geometry of the random polyhedron defined by the feasible region of the random LP \eqref{lp_max}, specifically about its spherical mean width. 
The computational experiments support our theoretical guarantees and give insights into several open questions regarding the asymptotic behaviour of the random LP. The connection between the proof of the main result and convex feasibility problems allows us to provide a fast algorithm for obtaining a near optimal solution in our randomized setting.
The random LP \eqref{lp_max}
is a step towards studying models with inhomogeneous coefficient matrices
and establishing stronger connections to practical applications.

\section*{Acknowledgements}
We are grateful to Dylan Altschuler for bringing our
attention to \cite[Section~3.7]{DZ1998}.

\section*{Declarations}

\noindent{\bf Funding.} Konstantin Tikhomirov
is partially supported by the NSF grant
DMS-2331037.

\noindent{\bf Conflict of interest/Competing interests.}
The authors declare no conflict of interest
in regard to this publication.

\noindent{\bf Code availability.}
The code for the numerical experiments is available on \url{https://github.com/marzb93/RandomLinearProgram}


\bibliographystyle{plain}
\bibliography{References}


\end{document}